\documentclass[12pt,a4paper]{amsart}
\usepackage{graphicx}
\usepackage{yhmath}
\usepackage{hyperref}
\usepackage{amsmath,verbatim}
\usepackage{amssymb}
\usepackage{pgf,tikz,pgfplots}
\pgfplotsset{compat=1.15}
\usetikzlibrary{arrows}
\usepackage[utf8]{inputenc}
\usepackage{mathrsfs}

\theoremstyle{plain}
\newtheorem{thm}{Theorem}[section]
\newtheorem{lemma}[thm]{Lemma}
\newtheorem{corollary}[thm]{Corollary}
\newtheorem{prop}[thm]{Proposition}
\newtoks\prt

\theoremstyle{definition}

\newtheorem{remark}[thm]{Remark}
\newtheorem{notation}[thm]{Notation}

\newtheorem{definition}[thm]{Definition}

\def\eqn#1$$#2$${\begin{equation}\label#1#2
\end{equation}}

\numberwithin{equation}{section}

\def\arc#1{\wideparen{#1}}

\def\diam{\operatorname{diam}}
\def\dist{\operatorname{dist}}

\def\bA{\textbf{A}}
\def\bB{\textbf{B}}
\def\bC{\textbf{C}}
\def\bD{\textbf{D}}
\def\bE{\textbf{E}}

\def\bH{\textbf{H}}
\def\bV{\textbf{V}}
\def\L{\mathcal{L}}

\def\cR{\mathcal{R}}
\def\P{\mathcal{P}}
\def\U{\mathcal{U}}

\def\H{\mathcal{H}}
\def\Q{\mathcal{Q}}

\def\bd{\textbf{d}}
\def\bX{\textbf{X}}
\def\bY{\textbf{Y}}
\def\bW{\textbf{W}}
\def\bZ{\textbf{Z}}

\def\phi{\varphi}
\def\epsilon{\varepsilon}
\def\eps{\varepsilon}
 
\def\N{\mathbb N}
\def\er{\mathbb R}
\def\R{\mathbb R}

\newtoks\by
\newtoks\paper
\newtoks\book
\newtoks\jour
\newtoks\yr
\newtoks\pages
\newtoks\vol
\newtoks\publ
\def\ota{{\hbox\vol{???}}}
\def\cLear{\by=\ota\paper=\ota\book=\ota\jour=\ota\yr=\ota
\pages=\ota\vol=\ota\publ=\ota}
\def\endpaper{\the\by, {\the\paper},
\textit{\the\jour} \textbf{\the\vol} (\the\yr), \the\pages.\cLear}
\def\endbook{\the\by, \textit{\the\book}, \the\publ.\cLear}
\def\endprep{\the\by, \textit{\the\paper}, \the\jour.\cLear}
\def\endyearprep{\the\by, \textit{\the\paper}, \the\jour, (\the\yr).\cLear}
\def\name#1#2{#2 #1}

\definecolor{ffqqqq}{rgb}{1.,0.,0.}
\definecolor{qqttqq}{rgb}{0.,0.2,0.}
\definecolor{ubqqys}{rgb}{0.29411764705882354,0.,0.5098039215686274}
\definecolor{wwqqzz}{rgb}{0.4,0.,0.6}
\definecolor{qqttzz}{rgb}{0.,0.2,0.6}
\definecolor{ttzzqq}{rgb}{0.2,0.6,0.}
\definecolor{yqqqqq}{rgb}{0.5019607843137255,0.,0.}
\definecolor{uuuuuu}{rgb}{0.26666666666666666,0.26666666666666666,0.26666666666666666}
\definecolor{ttttff}{rgb}{0.2,0.2,1.}
\definecolor{ffqqqq}{rgb}{1.,0.,0.}
\definecolor{qqccqq}{rgb}{0.,0.8,0.}
\definecolor{qqwuqq}{rgb}{0.,0.39215686274509803,0.}
\definecolor{qqabqq}{rgb}{0.,0.5,0.1}
\definecolor{xdxdff}{rgb}{0.49019607843137253,0.49019607843137253,1.}
\definecolor{qqqqcc}{rgb}{0.,0.,0.8}
\definecolor{ududff}{rgb}{0.30196078431372547,0.30196078431372547,1.}

\def\de0#1{\rule[3pt]{#1}{0.4pt} \hspace{-0.1pt} \rule[3.05pt]{0.05pt}{0.4pt} \hspace{-0.1pt} \rule[3.1pt]{0.05pt}{0.4pt} \hspace{-0.1pt} \rule[3.15pt]{0.05pt}{0.4pt} \hspace{-0.1pt} \rule[3.2pt]{0.05pt}{0.4pt} \hspace{-0.1pt} \rule[3.25pt]{0.05pt}{0.4pt} \hspace{-0.1pt} \rule[3.3pt]{0.05pt}{0.4pt} \hspace{-0.1pt} \rule[3.35pt]{0.05pt}{0.4pt} \hspace{-0.1pt} \rule[3.4pt]{0.05pt}{0.4pt} \hspace{-0.1pt} \rule[3.45pt]{0.05pt}{0.4pt} \hspace{-0.1pt} \rule[3.5pt]{0.05pt}{0.4pt} \hspace{-0.1pt} \rule[3.55pt]{0.05pt}{0.4pt} \hspace{-0.1pt} \rule[3.6pt]{0.05pt}{0.4pt} \hspace{-0.1pt} \rule[3.65pt]{0.05pt}{0.4pt} \hspace{-0.1pt} \rule[3.7pt]{0.05pt}{0.4pt} \hspace{-0.1pt} \rule[3.75pt]{0.05pt}{0.4pt} \hspace{-0.1pt} \rule[3.8pt]{0.05pt}{0.4pt} \hspace{-0.1pt} \rule[3.85pt]{0.05pt}{0.4pt} \hspace{-0.1pt} \rule[3.9pt]{0.05pt}{0.4pt} \hspace{-0.1pt} \rule[3.95pt]{0.05pt}{0.4pt} \hspace{-0.1pt} \rule[4.0pt]{0.05pt}{0.4pt} \hspace{-0.1pt} \rule[4.05pt]{0.05pt}{0.4pt} \hspace{-0.1pt} \rule[4.1pt]{0.05pt}{0.4pt} \hspace{-0.1pt} \rule[4.15pt]{0.05pt}{0.4pt} \hspace{-0.1pt} \rule[4.2pt]{0.05pt}{0.4pt} \hspace{-0.1pt} \rule[4.25pt]{0.05pt}{0.4pt} \hspace{-0.1pt} \rule[4.3pt]{0.05pt}{0.4pt} \hspace{-0.1pt} \rule[4.35pt]{0.05pt}{0.4pt} \hspace{-0.1pt} \rule[4.4pt]{0.05pt}{0.4pt} \hspace{-0.1pt} \rule[4.45pt]{0.05pt}{0.4pt} \hspace{-0.1pt} \rule[4.5pt]{0.05pt}{0.4pt} \hspace{-0.1pt} \rule[4.55pt]{0.05pt}{0.4pt} \hspace{-0.1pt} \rule[4.6pt]{0.05pt}{0.4pt} \hspace{-0.1pt} \rule[4.65pt]{0.05pt}{0.4pt} \hspace{-0.1pt} \rule[4.7pt]{0.05pt}{0.4pt} \hspace{-0.1pt} \rule[4.75pt]{0.05pt}{0.4pt} \hspace{-0.1pt} \rule[4.8pt]{0.05pt}{0.4pt} \hspace{-0.1pt} \rule[4.85pt]{0.05pt}{0.4pt} \hspace{-0.1pt} \rule[4.9pt]{0.05pt}{0.4pt} \hspace{-0.1pt} \rule[4.95pt]{0.05pt}{0.4pt} \hspace{-0.1pt} \rule[5.0pt]{0.05pt}{0.4pt} \hspace{-0.1pt} \rule[5.05pt]{0.05pt}{0.4pt} \hspace{-0.1pt} \rule[5.1pt]{0.05pt}{0.4pt} \hspace{-0.1pt} \rule[5.15pt]{0.05pt}{0.4pt} \hspace{-0.1pt} \rule[5.2pt]{0.05pt}{0.4pt} \hspace{-0.1pt} \rule[5.25pt]{0.05pt}{0.4pt} \hspace{-0.1pt} \rule[5.3pt]{0.05pt}{0.4pt} \hspace{-0.1pt} \rule[5.35pt]{0.05pt}{0.4pt} \hspace{-0.1pt} \rule[5.4pt]{0.05pt}{0.4pt} \hspace{-0.1pt} \rule[5.45pt]{0.05pt}{0.4pt} \hspace{-0.1pt} \rule[5.5pt]{0.05pt}{0.4pt} \hspace{-0.1pt} \rule[5.55pt]{0.05pt}{0.4pt} \hspace{-0.1pt} \rule[5.6pt]{0.05pt}{0.4pt} \hspace{-0.1pt} \rule[5.65pt]{0.05pt}{0.4pt} \hspace{-0.1pt} \rule[5.7pt]{0.05pt}{0.4pt} \hspace{-0.1pt} \rule[5.75pt]{0.05pt}{0.4pt} \hspace{-0.1pt} \rule[5.8pt]{0.05pt}{0.4pt} \hspace{-0.1pt} \rule[5.85pt]{0.05pt}{0.4pt} \hspace{-0.1pt} \rule[5.9pt]{0.05pt}{0.4pt} \hspace{-0.1pt} \rule[5.95pt]{0.05pt}{0.4pt} \hspace{-0.1pt} \rule[6.0pt]{0.05pt}{0.4pt}}	

\MakeRobust{\ref}

\makeatletter
\newcommand{\labeltext}[2]{%
	\@bsphack
	\def\@currentlabel{#1}{\label{#2}}%
	\@esphack
}
\makeatother

\def\step#1#2#3{\par \noindent{{\\ \bf Step~\labeltext{#1}{#3}#1. }{\bf #2. }}}

\begin{document}

\title[Minimal Extension for the $\alpha$-Manhattan norm]{Minimal Extension for the $\alpha$-Manhattan norm}

\author[D. Campbell]{Daniel Campbell}
\address{D.~Campbell: Department of Mathematics, University of Hradec Kr\' alov\' e, Rokitansk\'eho 62, 500 03 Hradec Kr\'alov\'e, Czech Republic} 
\email{daniel.campbell@uhk.cz}

\author[A. Kauranen]{Aapo Kauranen}
\address{A.~Kauranen: Department of Mathematics and Statistics, University of Jyv\"askyl\"a, PL 35, 40014 Jyv\"askly\"an yliopisto, Finland}
\email{aapo.p.kauranen@jyu.fi}

\author[E. Radici]{Emanuela Radici}
\address{E.~Radici: DISIM - Department of Information Engineering, Computer Science and Mathematics, University of L’Aquila, Via Vetoio 1 (Coppito), 67100 L’Aquila (AQ), Italy}
\email{emanuela.radici@univaq.it}

\thanks{The first author was supported by the grant GACR 20-19018Y}

\begin{abstract}
	Let $\partial \Q$ be the boundary of a convex polygon in $\er^2$, $e_\alpha = (\cos\alpha, \sin \alpha)$ and $e_{\alpha}^{\bot} = (-\sin\alpha , \cos \alpha)$ be a basis of $\er^2$ for some $\alpha\in[0,2\pi)$ and	$\phi:\partial\Q\to\er^2$ be a continuous, finitely piecewise linear injective map. We construct a finitely piecewise affine homeomorphism $v: \Q \to \er^2$ coinciding with $\phi$ on $\partial \Q$ such that the following property holds: $|\langle Dv, e_{\alpha}\rangle|(\Q)$ (resp. $\langle Dv, e_{\alpha}^{\bot}\rangle|(\Q)$) is as close as we want to $\inf |\langle Du, e_{\alpha}\rangle|(\Q)$ (resp. $\inf |\langle Du, e_{\alpha}^{\bot}\rangle|(\Q)$) where the infimum is meant over the class of all $BV$ homeomorphisms $u$ extending $\phi$ inside $\Q$. This result extends that already proven in \cite{PR2} in the shape of the domain.
\end{abstract}


\maketitle
\section{Introduction}

In this paper we are interested in the problem of extending injective continuous and piecewise linear boundary values from a convex polygon by piecewise affine homeomorphisms. The motivation for such a study arises in the context of approximation problems found in regularity theory for non-linear elasticity. There is already a plurality of extension results in a variety of contexts, which have been applied to solve various approximation problems. Let us now give an overview of some examples.

In general, we are interested in the approximation of a weakly differentiable homeomorphism, which we would like to approximate by $\mathcal{C}^1$ homeomorphisms or by locally finite piecewise affine homeomorphisms. The approximation of a planar $W^{1,p}$ homeomorphism $1<p<\infty$ in \cite{IKO1} and \cite{IKO2} relies heavily on the injectivity of the harmonic extension of convex boundary values. In \cite{DP} the authors were also able to approximate a bi-Lipschitz map and its inverse simultaneously in the $p,p$ bi-Sobolev setting and to do so used the extension result in \cite{DP2}. In order to solve the $W^{1,1}$ case in \cite{HP}, the authors had to develop an independent extension result in that paper which was further examined and improved in \cite{C} and \cite{R}. The extension result was also utilised in the $1,1$ bi-Sobolev setting in \cite{BiP}. Finally let us mention that the authors of \cite{PR3} approximate planar $BV$ homeomorphisms using an extension result they proved in \cite{PR2}.

More than just approximation of weakly differentiable homeomorphisms by diffeomorphisms these extension results have been key in examining the behaviour of weak and strong limits of homeomorphisms in their respective classes. Such results include a categorisation of the closure of $\operatorname{Hom} \cap W^{1,p}$, $p\geq2$ in \cite{IO}, a categorisation of the closure of $\operatorname{Hom} \cap W^{1,p}$, $1<p<2$ in \cite{PP} and partial $BV$ result in \cite{CKR}.

It was demonstrated in \cite{PR2} that their main extension result can be ``rotated'' to approximate a BV homeomorphism strictly and similarly in \cite{PR3} for the area-strict case. Nevertheless this approach makes the application of the extension result somewhat cumbersome and technical. The main result of the present paper is a piecewise affine homeomorphic extension that improves on that of \cite{PR2}. More precisely, we consider extensions of piecewise linear boundary values defined on boundary of convex quadrilaterals (and not only rectangles parallel to the coordinate axes as in \cite{PR2}) which are optimal in a particular BV sense. We emphasize that the generality of the class of convex quadrilaterals includes the ``rotated'' version of the extension result of \cite{PR2}. Also our Theorem~\ref{Componentwise ExtensionAlt} is stronger than the extension theorem there (not only because of the shape of $\Q$) in the sense that it immediately implies their extension theorem but the opposite is not true. Nevertheless this improvement is a case of separating estimates already conducted in \cite{PR2}.

The motivation for our extension theorem is the full categorisation result in \cite{CKR2}, where we identify a condition which guarantees that a map is a strict or area-strict limit of BV homeomorphisms. In the course of the approximation we want to work on grids that are not only made up of rectangles and we prefer to not have to rotate the rectangles. In that sense we need the current result, which we present below, after we set some necessary notation.

Let $\Q\subset \er^2$ be a convex polygon, let $\alpha\in [0,2\pi)$ be fixed and call $e_\alpha = (\cos\alpha, \sin \alpha)$ and $e_{\alpha}^{\perp} = (-\sin\alpha ,  \cos \alpha)$. 
We define the following numbers
\begin{equation}\label{Useful1}
		\begin{aligned}
			&a^-:=\inf\{\langle x,e_{\alpha}^{\bot}\rangle : x\in \Q\} \qquad & a^+:=\sup\{\langle x,e_{\alpha}^{\bot}\rangle :  x\in \Q\},\\
			&b^-:=\inf\{\langle x,e_{\alpha}\rangle : x\in \Q\},\qquad & b^+ := \sup\{\langle x,e_{\alpha}\rangle : x\in \Q\}.
		\end{aligned}
\end{equation}
For each $s \in (a^-, a^+)$ we define $V_s^1, V_s^2$ uniquely by the conditions
\begin{equation}\label{Useful2}
	V_s^1, V_s^2\in \partial \Q, \quad  \langle V_s^{1}, e_{\alpha}^{\bot} \rangle = \langle V_s^{2}, e_{\alpha}^{\bot} \rangle = s, \quad  \langle V_s^{1}, e_{\alpha} \rangle < \langle V_s^{2}, e_{\alpha} \rangle.
\end{equation}
Similarly for every $t \in (b^-, b^+)$ we define $H_t^1, H_t^2$ uniquely by
\begin{equation}\label{Useful}
	H_t^1, H_t^2\in \partial \Q, \quad  \langle H_t^{1}, e_{\alpha} \rangle = \langle H_t^{2}, e_{\alpha} \rangle = t, \quad  \langle H_t^{1}, e_{\alpha}^{\bot} \rangle < \langle H_t^{2}, e_{\alpha}^{\bot} \rangle.
\end{equation}

Let $\phi: \partial \Q \to \R^2$ be continuous, injective and piecewise linear. We denote $\P$ as the bounded component of $\R^2 \setminus \phi(\partial \Q)$. For every pair of points $\bA,\bB \in \overline{\P}$ we denote by $\rho_{\P}(\bA,\bB)$ the geodesic distance between $\bA$ and $\bB$ inside $\overline{\P}$. 
We define the quantity
$$
	\Psi_\alpha (\varphi):= \int_{a^-}^{a^+} \rho_\P (\phi(V_s^1), \phi(V_s^2)) ds + \int_{b^-}^{b^+} \rho_\P (\phi(H_t^1), \phi(H_t^2)) dt.
$$
Further for $u \in BV(\Omega,\R^2)$ we denote the $\alpha$-Manhattan norm of $Du$ as $\| \cdot \|_\alpha$ which we define as
$$
	\| Du \|_\alpha (\Q) :=  |\langle D u, e_\alpha\rangle|(\Q) + |\langle D u, e_\alpha^{\bot}\rangle|(\Q).
$$
The main results of the paper is are the follwoing.

\begin{thm}\label{thm: rotated minimal extension}
	Let $\alpha \in [0,2\pi)$ be fixed, $\Q \subset \R^2$ be a convex polygon and $\phi: \partial \Q \to \R^2$ be a continuous piecewise linear injective map.  Then for every $\eps >0$ there exists a finitely piecewise affine homeomorphism $v: \Q \to \R^2$ extending $\phi$, such that 
	\begin{equation}\label{eq: rotated minimal extension}
	\| Dv \|_\alpha (\Q) \leq \Psi_\alpha(\phi) + \eps. 
	\end{equation}
\end{thm}

\begin{thm}\label{Componentwise ExtensionAlt}
	Let $\epsilon>0$ and let $v$ be the extension from Theorem~\ref{thm: rotated minimal extension} then
	\begin{equation}\label{eqComponentwise extensionAlt}
	\begin{aligned}
	| \langle Dv , e_{\alpha}\rangle|(\Q) &\leq\int_{b^-}^{b^+} \rho_\P (\phi(H_t^1), \phi(H_t^2)) dt + \epsilon \text{ and }\\
	| \langle Dv , e_{\alpha}^{\bot}\rangle|(\Q) &\leq  \int_{a^-}^{a^+} \rho_\P (\phi(V_s^1), \phi(V_s^2)) ds +\epsilon. 
	\end{aligned}
	\end{equation}
\end{thm}

Let us remark that Theorem~\ref{Componentwise ExtensionAlt} immediately implies Theorem~\ref{thm: rotated minimal extension} but the argument used to construct $v$ is exactly the same. Also it is immediate that
$$
\int_{b^-}^{b^+} \rho_\P (\phi(H_t^1), \phi(H_t^2)) dt \leq \inf\big\{ |\langle Du, e_{\alpha}\rangle|(\overline{\Q}): u\in \operatorname{Hom}\cap BV(\overline{\Q},\er^2), u=\varphi \text { on }\partial \Q  \big\}
$$
and
$$
\int_{a^-}^{a^+} \rho_\P (\phi(V_s^1), \phi(V_s^2)) ds\leq\inf\big\{ |\langle Du, e_{\alpha}^{\bot}\rangle|(\overline{\Q}): u\in \operatorname{Hom}\cap  BV(\overline{\Q},\er^2), u=\varphi \text { on }\partial \Q  \big\}
$$
and our result in fact shows that there is a sequesnce of homeomorphisms achieving the infimum and having variation converging to the left hand side in the sense of \eqref{eqComponentwise extensionAlt}. This fact is actually a direct consequence of the proofs in \cite{PR2}, though it was not explicitely remarked there. The key argument is in Theorem~\ref{thm:What}.

\subsection{Sketch of the proof}

Before expounding the proof in detail, let us look at an overview of the proof. We start with a convex polygon $\Q$. Up to a rotation of $\alpha$ we may assume that $\alpha = 0$. Either (the rotated) $\Q$ has horizontal sides, or after removing a tiny triangle called $T_1$ close to the lowest point of $\Q$ and a triangle called $T_2$ close to the highest point of $\Q$ we get a convex $\Delta$ that has a pair of horizontal sides (see Figure~\ref{Fig:Setup}). We extend $\phi$ on $\partial T_1, \partial T_2$ so that it is continuous injective and piecewise linear. By making the triangles small enough we guarantee that $\Psi_0(\phi(\Delta)) +\Psi_0(\phi(T_1))+\Psi_0(\phi(T_2)) \leq \Psi_0(\phi(\Q))+ \eps$. Here our new $\phi$ extends the original $\phi$ from $\partial\Q$. This step is Lemma~\ref{lemma: skeleton triangles}.

Now we separate $\Delta$ into thin horizontal strips $S_i$ (see Figure~\ref{Fig:Slicing}), defining a continuous injective piecewise linear $\phi$ on $\partial S_i$ so that $\sum_{i=1}^M\Psi_0(\phi(\partial S_i)) \leq \Psi_0(\phi(\partial \Delta)) + \eps$ which extends the original $\phi$ from $\partial\Delta \cup \partial T_1 \cup \partial T_2$. This step is Lemma~\ref{lemma: skeleton strips}.

We separate each $S_i$ into a central rectangle and a pair of right-angle triangles at each end. On the rectangular domains $R_i$ we can use Proposition~\ref{prop: straight minimal extension} to extend the boundary values and get a piecewise affine homeomorphism $w_i$ on the $R_i$ satisfying an estimate on $|Dw_i|(R_i)$. In Lemma~\ref{lemma: extension left & right triangles} we show how we extend the boundary values to get a piecewise affine homeomorphism on the triangles at the ends of the strips, see Figure~\ref{Fig:Conquer}. We do this by further separating them into even thinner rectangles where we can extend and estimate as above. The remaining part of the set is so small that its contribution to the norm is bounded by $2^{-i}\eps$.

The final part of the proof is collating the estimates and summing to estimate that our mapping $v$ satisfies \eqref{eq: rotated minimal extension}.

\section{Preliminaries}

In this subsection we recall a list of definitions and known geometrical results which are already available in the literature.  Most of them are taken from \cite{PR2} and \cite{PR3}. 

\begin{notation}
	Throughout the paper we endeavour to keep to the following norms of notation:
	\begin{itemize}
		\item[$\cdot$] $\Q$ is a convex polygon,
		\item[$\cdot$] $\alpha\in[0,2\pi)$ is a given angle and the vector $e_{\alpha}:=(\cos\alpha, \sin\alpha)$. Also we denote $e_{\alpha}^{\bot}:=(\cos(\alpha+\pi /2), \sin(\alpha+\pi /2))$,
		\item[$\cdot$] $u$ and $v$ are planar $BV$ mappings,
		\item[$\cdot$] $a^-, a^+, b^-, b^+$ are the numbers from \eqref{Useful1}, typically $s \in (a^-, a^+)$ and $t \in (b^-, b^+)$ and $\ell = a^+ - a^-$, $h = b^+ - b^-$,
		\item[$\cdot$] $\Delta$ is a convex polygon with 2 sides parallel to $\alpha$,
		\item[$\cdot$] $T, T_1, T_2, \tilde{T}, T_i^1, T_i^2$ are all triangles,
		\item[$\cdot$] $V_s^1, V_s^2, H_t^1, H_t^2$ are the points satisfying the conditions in \eqref{Useful2} and \eqref{Useful} although we may replace $\Q$ with another convex polygon, for example $\Delta$ or $T$,
		\item[$\cdot$] by $\cR_{\Q} = [a^-,a^+]e_{\alpha} + [b^{-}, b^+]e_{\alpha}^{\bot}$ we denote the smallest rectangle with sides parallel to $e_{\alpha}$ and $e_{\alpha}^{\bot}$ containing $\Q$,
		\item[$\cdot$] $c_1, c_2, d_1, d_2, \in \R$ are ordinates,
		\item[$\cdot$] by $\tilde{C}$ we denote a generic constant whose precise value may vary between estimates,
		\item[$\cdot$] points in the preimage are $A,B,C, D,E,F,G, P,Q$,\footnote{We do not need to utilise the notation $B(x,r) = \{y: |y-x|<r\}$ so there is no danger of confusion when using $B$ to denote a point.}
		\item[$\cdot$] $\bd := |\varphi(A) - \varphi(C)|$ is the length of the image of the hypotenuse of the triangle $ABC$ in $\varphi$,
		\item[$\cdot$] $\beta\in (0,\frac{\pi}{2})$ is the angle at $A$ in the triangle $ABC$,
		\item[$\cdot$] $\eta>0$ is a small chosen parameter,
		\item[$\cdot$] points in the image are written in bold font e.g. $\bA, \bB, \bC, \bD, \bX,\bY,\bZ$,
		\item[$\cdot$] $\P, \P_0, \P^+$ are polygons in the image, typically the piecewise affine image of a polygon in the preimage e.g. $\P = \phi(\partial \Q)$,
		\item[$\cdot$] $\phi$, $\psi$ are continuous injective piecewise linear maps from one dimensional `skeletons' (i.e. a finite union of segments) in the preimage,
		\item[$\cdot$] $\gamma_{\bA, \bB}$ is the geodesic curve from $\bA$ to $\bB$ in $\P$ and $\rho_{\P}(\bA,\bB)$ is the length of that curve.
	\end{itemize}
\end{notation}

\begin{remark}[Geodesics and modified geodesics]\label{def: geodesics and modified geodesics}
Let $\P \subset \R^2$ be a polygon, and let $\bA$ and $\bB$ be any two distinct points in $\P$.  We define $\gamma_{\bA\bB}$ the unique geodesic (i.e., curve of minimal length) connecting them, lying inside $\P$. Notice that $\gamma_{\bA\bB}$ is a piecewise linear curve, whose vertices are only $\bA, \bB$ and some vertices of $\partial \P$ whose internal angles have size at least $\pi$. Assume now that $\bA, \bB \in \partial \P$, and let
$\bW1, \bW_2,  .\ldots, \bW_K$ be all the vertices of $\P$ met by $\gamma_{\bA\bB}$, so that $\gamma_{\bA\bB} = \bA \bW_1, \bW_2,  \ldots, \bW_K \bB$.
Fix now any $\delta > 0$. For every $1 \leq i \leq K$ let $\widetilde{\bW}_i \neq \bW_i$ be some   
arbitrary point in the internal bisector of the angle at $\bW_i$ having distance from $\bW_i$ smaller than $\delta$.
The piecewise linear curve $\tilde{\gamma}_{\bA\bB} =  \bA \widetilde{\bW}_1,  \widetilde{\bW}_2,  \ldots,  \widetilde{\bW}_K \bB$ is then called a
\emph{$\delta$-modification} of $\gamma_{\bA\bB}$.
\end{remark}

Notice that there exists a constant $\bar{\delta}(\P) > 0$,  depending on $\P$ but not on $\bA$ and $\bB$, such that the interior of $\tilde{\gamma}_{\bA\bB}$ is contained in the interior of $\P$ if $\delta < \bar{\delta}(\P)$, unless the segment $\bA\bB$ is already contained in $\partial \P$, in which case $K = 0$ and $\tilde{\gamma}_{\bA\bB} = \gamma_{\bA\bB} \subseteq \partial \P$.

\begin{lemma}[\cite{PR2}, Lemma 2.4]\label{lemma: 2.4 Pratelli2}
Let $\bA, \bB, \bC$ and $\bD$ be four distinct points in a polygon $\P$. Then the intersection $\gamma_{\bA\bB} \cap \gamma_{\bC\bD}$ is either empty or closed and connected.  Assume now also that $\bA, \bB, \bC, \bD\in \partial \P$ and call $\P_1, \P_2$ the two components of $\P\setminus \{ \bC, \bD\}$. If $\bA\in \P_1$ and $\bB\in \P_2$ then $\gamma_{\bA\bB}\cap \gamma_{\bC\bD}\neq \emptyset$. If $\bA, \bB\in \P_1$ and $\gamma_{\bA\bB}\cap \gamma_{\bC\bD}\neq \emptyset$ then the first and last point of this intersection must be either vertices of $\P$ or coincide with one of the points $\bA$ or $\bB$.
\end{lemma}

\begin{lemma}[\cite{PR2},  Lemma 2.5]\label{lemma: 2.5 Pratelli2}
Let $\P$ be a polygon, let $\bA,\bB \in \partial \P$ be two points such that the segment $\bA\bB$ is not contained in $\partial \P$, then let $\delta < \bar{\delta}(\P)$ and let $\tilde{\gamma}_{\bA\bB}$ be a modified geodesic in
the sense of Definition \ref{def: geodesics and modified geodesics}.  Let also $\P_1$ and $\P_2$ be the two polygons in which $\P$ is divided by $\tilde{\gamma}_{\bA\bB}$,  and let $\eps > 0$ be a given constant. 
If $\delta$ is small enough, depending only on $\eps$ and $\P$, then the following is true:
\begin{enumerate}
\item[] For any two points $\bC,\bD \in \P_i$ for $i \in  \{1,2\}$ one has 
\begin{equation}\label{eq: stesso poligono}
\rho_{\P_i}(\bC,\bD) <   \rho_{\P}(\bC,\bD) + \eps
\end{equation}
\item[] If $\bC \in \P_1$, $\bD \in \P_2$ and $\bE \in \partial \P_1 \cap \partial \P_2$ is any point with distance at most $\delta$ from $\gamma_{\bC\bD}$, then
\begin{equation}\label{eq: diversi poligon}
 \rho_{\P_1}(\bC,\bE) + \rho_{\P_2}(\bE,\bD) < \rho_\P (\bC,\bD) + \eps. 
\end{equation} 
\end{enumerate}
\end{lemma}

\begin{definition}[Set of vertices of a geodesic curve]\label{def: definizione 2.6 Pratelli}
Let $\P \subset \R^2$ be a polygon.  For every $\bA, \bB \in \partial \P$ there is a unique ordered set $\mathcal{X}(\bA,\bB)= \{ \bX_1, \ldots, \bX_N  \}$ such that the geodesic $\gamma_{\bA\bB}$ is exactly the piecewise linear curve $\bA \bX_1 \ldots \bX_N \bB$, and the points $\bX_j$ are all the vertices of $\P$ met by the geodesic $\gamma_{\bA\bB}$ (except $\bA$ and $\bB$ themselves, in case they are already vertices). The set $\mathcal{X}(\bA,\bB)$ is called \emph{set of vertices} of $\gamma_{\bA\bB}$. 
\end{definition}

\begin{definition}[$\delta$-linearization of a Jordan curve]\label{def: 3.6 Pratelli3}
Let $\psi$ be a Jordan curve with finite length, and let $\delta >0$ be much smaller than the diameter of the bounded component of $\R^2 \setminus\psi$.  Let $\arc{\bA_1 \bB_1},  \arc{\bA_2 \bB_2}, \ldots,  \arc{\bA_N \bB_N}$ be finitely many essentially disjoint arcs contained in $\psi$.
Let then $\phi$ be the closed curve obtained by replacing each arc $\arc{\bA_i \bB_i}$ with the segment $\bA_i \bB_i$. We say that $\phi$ is a $\delta$-linearization of $\psi$ if 
\begin{itemize}
	\item $\phi$ is injective
	\item  every arc $\arc{\bA_i \bB_i}$ is such that $\mathcal{H}^1(\arc{\bA_i \bB_i}) < \delta$
	\item   $\arc{\bA_i \bB_i} \cap \phi \subset \bA_i \bB_i$.
\end{itemize}  
The $\delta$-linearization is said \emph{complete} if the union of the arcs $\arc{\bA_i \bB_i}$ is the whole curve $\psi$, hence $\phi$ is piecewise linear. 
\end{definition}

\begin{lemma}[\cite{PR2}, Corollary 4.3]\label{lemma: 3.7 Pratelli3}
Let $\Delta \subset \R^2$ be a convex polygon and let $\psi: \partial \Delta \to \R^2$ be a parametrized Jordan curve with finite length and $\phi: \partial \Delta \to \R^2$ be a $\delta$-linearization of $\psi$.
Then for every $P,Q \in \partial \Delta$ one has 
\[  \rho_{\phi(\partial \Delta)} \big( \phi(P), \phi(Q) \big) \leq \rho_{\psi(\partial \Delta)} \big( \psi(P), \psi(Q) \big) + 2\delta.  \]
In particular for every $\theta \in [0,2\pi)$ we deduce
\[ \Psi_\theta (\phi) \leq \Psi_\theta (\psi) + 2\delta \mathcal{H}^1(\partial \Delta).   \]
\end{lemma}

We conclude this subsection recalling two extension results that will be useful in the sequel. The next Proposition \ref{prop: straight minimal extension} is proved in \cite{PR2} Theorem A,
and Corollary \ref{coroll: non optimal extension} is a straightforward consequence of Proposition \ref{prop: straight minimal extension}. 

\begin{prop}[Minimal extension for standard Manhattan norm]\label{prop: straight minimal extension}
Let $\cR \subset \R^2$ be a rectangle of the form $[a^-,a^+] \times [b^-,b^+]$ and let $\phi: \partial \cR \to \R^2$ be a continuous injective map.  Then for every $\eps > 0 $ there exists a piecewise affine homeomorphism $v : \cR \to \R^2$ coinciding with $\phi$ on $\partial \cR$ such that
\begin{equation}\label{eq:straight minimal extension}
\| Dv\|_0 (\cR) \leq \Psi_0 (\phi) + \eps. 
\end{equation} 
Moreover, if $\phi$ is piecewise linear then the map $v$ can be chosen finitely piecewise affine.
\end{prop}

\begin{thm}[Minimal extension for standard Manhattan norm]\label{thm:What}
	Let $\epsilon > 0$ and let $v$ be the mapping from Proposition~\ref{prop: straight minimal extension} then
	\begin{equation}\label{eq:Component Straight minimal extension}
	\begin{aligned}
	|D_1v|(\Q) &\leq\int_{b^-}^{b^+} \rho_\P (\phi(H_t^1), \phi(H_t^2)) dt + \epsilon \text{ and }\\
	|D_2v|(\Q) &\leq  \int_{a^-}^{a^+} \rho_\P (\phi(V_s^1), \phi(V_s^2)) ds + \epsilon . 
	\end{aligned}
	\end{equation}
\end{thm}
\begin{proof}
	The finitely piecewise affine homeomorphisms from a rectangle to a polygon in \cite{PR2} used in the proof of Proposition~\ref{prop: straight minimal extension} are constructed in Lemma~2.12. The key estimates we need to extract are the last two unnumbered equations of the proof, found on page 543. They say exactly that
	$$
	\begin{aligned}
	|D_1v|(\Q) &\leq\int_{b^-}^{b^+} \rho_\P (\phi(H_t^1), \phi(H_t^2)) dt+ \epsilon \text{ and }\\
	|D_2v|(\Q) &\leq  \int_{a^-}^{a^+} \rho_\P (\phi(V_s^1), \phi(V_s^2)) ds + \epsilon . 
	\end{aligned}
	$$
\end{proof}

\begin{corollary}[$W^{1,1}$ extension with non optimal bound]\label{coroll: non optimal extension}
There exists a $\tilde{C}>0$ such that the following holds. Let $\cR \subset \R^2$ be a rectangle and let $\phi: \partial \cR \to \R^2$ be a continuous, piecewise linear and injective map.  Then there exists a finitely piecewise affine homeomorphism $v : \cR \to \R^2$ extending $\phi$ such that 
\begin{equation}
 \| Dv  \|_{L^1(\cR)} \leq \tilde{C} \mathcal{H}^1(\partial \cR) \mathcal{H}^1(\phi(\partial \cR))
\end{equation}
\end{corollary}
\begin{proof}
The conclusion follows by applying Proposition \ref{prop: straight minimal extension} with $\eps =  \mathcal{H}^1(\partial \cR) \mathcal{H}^1(\phi(\partial \cR))$ and by observing that 
\[   \| Dv  \|_{L^1(\cR)} \leq \tilde{C}\big(|D_1 v|(\cR) +  |D_2 v|(\cR)\big) =  \tilde{C}\| Dv\|_0 (\cR)  \]
and 
\begin{align*}
\Psi_0 (\phi) &= \int_{a^-}^{a^+} \rho_{\phi(\partial cR)} (\phi(t,b^-), \phi(t,b^+)) dt + \int_{b^-}^{b^+} \rho_{\phi(\partial cR)} (\phi(a^-,t), \phi(a^+,t)) dt \\
&\leq \big(a^+ - a^- + b^+ - b^-   \big) \mathcal{H}^1 (\phi(\partial \cR)) \\
&\leq \mathcal{H}^1(\partial \cR) \mathcal{H}^1(\phi(\partial \cR)).
\end{align*}
\end{proof}

\section{Extension on a one-dimensional skeleton}
\begin{remark}\label{rmk: class I) and II)}
	If $\Q$ coincides with the rectangle $\cR_\Q$ then the conclusion of Theorem \ref{thm: rotated minimal extension} follows directly by \cite[Theorem A]{PR2}.  So,  without loss of generality,  we can assume that $\Q \neq \cR_\Q$. Then there are three cases left to consider. Either
	\begin{itemize}
		\item[$i)$] $\Q$ has two parallel sides parallel to $\alpha$
		\item[$ii)$] $\Q$ has exactly one side parallel to $\alpha$
		\item[$iii)$] $\Q$ has no side parallel to $\alpha$.
	\end{itemize}
\end{remark}

In this subsection we introduce a suitable partition of $\Q$, whose boundary will be referred to as one-dimensional \emph{skeleton}, and we define a continuous, injective and piecewise linear extension of $\phi$ on the skeleton. This procedure will be done in two different steps, which eventually correspond to the following technical lemmas. 

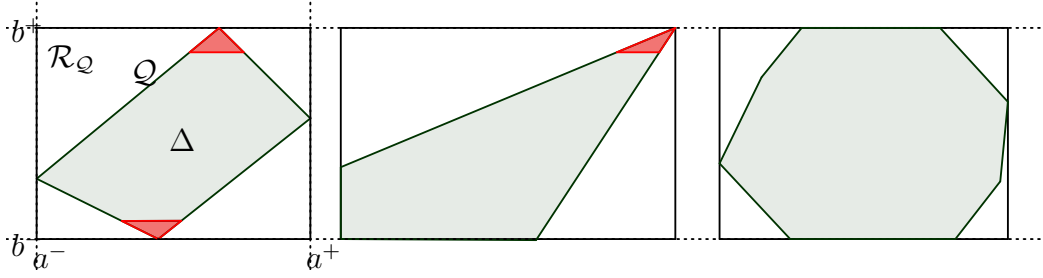
\begin{figure}
	\begin{tikzpicture}[line cap=round,line join=round,>=triangle 45,x=0.4cm,y=0.4cm]
	\clip(1.,-1.) rectangle (35.,8.);
	\fill[line width=0.7pt,color=qqttqq,fill=qqttqq,fill opacity=0.10000000149011612] (2.,2.) -- (8.,7.) -- (11.,4.) -- (6.,0.) -- cycle;
	\fill[line width=0.7pt,color=qqttqq,fill=qqttqq,fill opacity=0.10000000149011612] (12.,0.) -- (12.,2.3835369673354077) -- (23.,7.) -- (18.420510218271076,-0.034892494895193644) -- cycle;
	\fill[line width=0.7pt,color=ffqqqq,fill=ffqqqq,fill opacity=0.5] (4.8,0.6) -- (6.,0.) -- (6.757258216070825,0.6058065728566602) -- cycle;
	\fill[line width=0.7pt,color=ffqqqq,fill=ffqqqq,fill opacity=0.5] (22.473898109620652,6.191816028295064) -- (21.074278748682115,6.191816275643125) -- (23.,7.) -- cycle;
	\fill[line width=0.7pt,color=qqttqq,fill=qqttqq,fill opacity=0.10000000149011612] (27.153829703201264,7.) -- (25.835356906871976,5.364347226300464) -- (24.454099691669867,2.5076561675869726) -- (26.777123189964325,0.) -- (32.20797542246353,0.) -- (33.68340926597487,1.9112041882951458) -- (33.93454694146617,4.548149780953753) -- (31.705700071480944,7.) -- cycle;
	\fill[line width=0.7pt,color=ffqqqq,fill=ffqqqq,fill opacity=0.5] (8.,7.) -- (7.030182509107209,6.191818757589341) -- (8.80818155662796,6.19181844337204) -- cycle;
	\draw [line width=0.7pt,color=qqttqq] (2.,2.)-- (8.,7.);
	\draw [line width=0.7pt,color=qqttqq] (8.,7.)-- (11.,4.);
	\draw [line width=0.7pt,color=qqttqq] (11.,4.)-- (6.,0.);
	\draw [line width=0.7pt,color=qqttqq] (6.,0.)-- (2.,2.);
	\draw [line width=0.7pt] (2.,0.)-- (2.,7.);
	\draw [line width=0.7pt] (2.,7.)-- (11.,7.);
	\draw [line width=0.7pt] (11.,7.)-- (11.,0.);
	\draw [line width=0.7pt] (11.,0.)-- (2.,0.);
	\draw [line width=0.7pt] (4.8,0.6)-- (6.757258216070825,0.6058065728566602);

	\draw [line width=0.7pt] (12.,0.)-- (12.,7.);
	\draw [line width=0.7pt] (12.,7.)-- (23.,7.);
	\draw [line width=0.7pt] (23.,7.)-- (23.,0.);
	\draw [line width=0.7pt] (23.,0.)-- (12.,0.);
	\draw [line width=0.7pt,color=qqttqq] (12.,0.)-- (12.,2.3835369673354077);
	\draw [line width=0.7pt,color=qqttqq] (12.,2.3835369673354077)-- (23.,7.);
	\draw [line width=0.7pt,color=qqttqq] (23.,7.)-- (18.420510218271076,-0.034892494895193644);
	\draw [line width=0.7pt,color=qqttqq] (18.420510218271076,-0.034892494895193644)-- (12.,0.);
	\draw [line width=0.7pt,color=ffqqqq] (4.8,0.6)-- (6.,0.);
	\draw [line width=0.7pt,color=ffqqqq] (6.,0.)-- (6.757258216070825,0.6058065728566602);
	\draw [line width=0.7pt,color=ffqqqq] (6.757258216070825,0.6058065728566602)-- (4.8,0.6);

\draw (1.5,0.2) node[anchor=north west] {$a^-$};
\draw (10.5,0.2) node[anchor=north west] {$a^+$};
\draw (0.8,0.6) node[anchor=north west] {$b^-$};
\draw (0.8,7.7) node[anchor=north west] {$b^+$};
\draw (4.737737588284016,6.3) node[anchor=north west] {$\Q$};
\draw (2,6.8) node[anchor=north west] {$\cR_{\Q}$};
\draw (6,4) node[anchor=north west] {$\Delta$};

	\draw [line width=0.7pt,dotted] (2.,-1.) -- (2.,8.);
	\draw [line width=0.7pt,dotted] (11.,-1.) -- (11.,8.);
	\draw [line width=0.7pt,dotted,domain=1.:35.] plot(\x,{(--21.-0.*\x)/3.});
	\draw [line width=0.7pt,dotted,domain=1.:35.] plot(\x,{(-0.-0.*\x)/9.});
	\draw [line width=0.7pt,color=ffqqqq] (22.473898109620652,6.191816028295064)-- (21.074278748682115,6.191816275643125);
	\draw [line width=0.7pt,color=ffqqqq] (21.074278748682115,6.191816275643125)-- (23.,7.);
	\draw [line width=0.7pt,color=ffqqqq] (23.,7.)-- (22.473898109620652,6.191816028295064);
	\draw [line width=0.7pt] (24.454099691669867,7.)-- (24.454099691669867,0.);
	\draw [line width=0.7pt] (24.454099691669867,0.)-- (33.93454694146617,0.);
	\draw [line width=0.7pt] (33.93454694146617,0.)-- (33.93454694146617,7.);
	\draw [line width=0.7pt] (33.93454694146617,7.)-- (24.454099691669867,7.);
	\draw [line width=0.7pt,color=qqttqq] (27.153829703201264,7.)-- (25.835356906871976,5.364347226300464);
	\draw [line width=0.7pt,color=qqttqq] (25.835356906871976,5.364347226300464)-- (24.454099691669867,2.5076561675869726);
	\draw [line width=0.7pt,color=qqttqq] (24.454099691669867,2.5076561675869726)-- (26.777123189964325,0.);
	\draw [line width=0.7pt,color=qqttqq] (26.777123189964325,0.)-- (32.20797542246353,0.);
	\draw [line width=0.7pt,color=qqttqq] (32.20797542246353,0.)-- (33.68340926597487,1.9112041882951458);
	\draw [line width=0.7pt,color=qqttqq] (33.68340926597487,1.9112041882951458)-- (33.93454694146617,4.548149780953753);
	\draw [line width=0.7pt,color=qqttqq] (33.93454694146617,4.548149780953753)-- (31.705700071480944,7.);
	\draw [line width=0.7pt,color=qqttqq] (31.705700071480944,7.)-- (27.153829703201264,7.);
	\draw [line width=0.7pt,color=ffqqqq] (8.,7.)-- (7.030182509107209,6.191818757589341);
	\draw [line width=0.7pt,color=ffqqqq] (7.030182509107209,6.191818757589341)-- (8.80818155662796,6.19181844337204);
	\draw [line width=0.7pt,color=ffqqqq] (8.80818155662796,6.19181844337204)-- (8.,7.);
	\end{tikzpicture}
	\caption{The figure shows the set-up in the case that $\alpha = 0$. We have the polygon $\Q$ constituted of the green set $\Delta$ and the removed red triangle(s). In the case where the polygon $\Q$ has no sides parallel to $\alpha$ (corresponding to class $iii$ from Remark~\ref{rmk: class I) and II)} pictured on the left) we generate a polygon with two sides, both parallel to $\alpha$ by removing 2 triangles. If there is already one side parallel to some $\alpha$ (as pictured in the middle) then it is enough to remove one triangle and the remaining set has two sides both parallel to $\alpha$; this corresponds to class $ii$ from Remark~\ref{rmk: class I) and II)}. In the scenario on the right $\Q$ already has two sides parallel to $\alpha$ (corresponding to class $i$ from Remark~\ref{rmk: class I) and II)}), it is not necessary to remove any triangles.}
	\label{Fig:Setup}
\end{figure}

\begin{lemma}[Skeleton- triangles]\label{lemma: skeleton triangles}
Let $\alpha \in [0,2\pi)$ be fixed, $\Q \subset \R^2$ be a convex polygon and let $\phi: \partial \Q \to \R^2$ be a continuous, piecewise linear, injective map.  

Then for every $\eps >0$ one of the following holds:

\begin{itemize}
\item the set $\Q$ is of class $ii)$ from Remark~\ref{rmk: class I) and II)} and there exists a triangle $T$ satisfying
\begin{itemize}
	\item[a)] two sides of $T$ are contained in two sides of $\Q$, so $T$ and $\Q$ share the vertex $W$ and the side of $T$ inside $\Q$ is parallel to $\alpha$ (we refer to the point in the intersection of the third side of $T$ and the bisector of the vertex at $W$ as $W^*$),
	\item[b)] $\mathcal{H}^1(\partial T)  < \eps$
	\item[c)] $\Q = T \cup \Delta$, where $\partial\Delta$ is a convex polygon with two sides parallel to $\alpha$ one of which lies in $\partial \Q$ and the other one is the common side of $\partial T$ and $\partial\Delta$
	\item[d)] $\phi$ is linear on each side of $\partial T \cap \partial \Q$
	\item[e)] there exists a $\bar{\phi} : \partial T \cup \partial \Delta \to \R^2$ a continuous piecewise linear injective map such that $\bar{\phi} = \phi$ on $\partial\Q$, $\bar{\phi}$ is exactly bi-linear 
	on $\partial T \cap \partial \Delta$ and the singular point is precisely $W^*$
	\item[f)] the estimates
	$$
	\begin{aligned}
	\mathcal{H}^1(\bar{\phi}(\partial T)) &\leq \eps, \\
	\Psi_\alpha(\bar{\phi}_{\rceil \partial T}) + \Psi_\alpha(\bar{\phi}_{\rceil \partial \Delta})  &< \Psi_\alpha(\phi) + \eps
	\end{aligned}
	$$
	hold.
\end{itemize} 

\item the set $\Q$ is of class $iii)$ from Remark~\ref{rmk: class I) and II)} and there exists a pair of disjoint triangles $T_1, T_2$ satisfying
\begin{itemize}
	\item[a)] $T_1,T_2$ each contain a vertex of $\Q$ which we call $W_1$ resp. $W_2$, 
	\item[b)] It holds that
	\begin{equation}\label{SmallTry}
		\mathcal{H}^1(\partial T_1), \mathcal{H}^1(\partial T_2)  < \eps
	\end{equation}
	\item[c)] $\Q = T_1 \cup T_2 \cup \Delta$, where $\partial\Delta$ is a convex polygon with two sides parallel to $\alpha$ one of which is the common side of $\partial T_1$ and $\partial\Delta$ and the other is the common side of $\partial T_2$ and $\partial\Delta$
	\item[d)] $\phi$ is linear on each side of $\partial T_1 \cap \partial \Q$ and $\partial T_2 \cap \partial \Q$
	\item[e)] there exists a $\bar{\phi} : \partial T_1 \cup \partial T_2 \cup \partial \Delta \to \R^2$ a continuous piecewise linear injective map such that $\bar{\phi} = \phi$ on $\partial\Q$, $\bar{\phi}$ is exactly bi-linear 
	on $\partial T_2 \cap \partial \Delta$ and on $\partial T_1 \cap \partial \Delta$ and the singular points are precisely $W^*_1, W^*_2$
	\item[f)] the estimates
	\begin{equation}\label{SmallImageTry}
		\mathcal{H}^1(\bar{\phi}(\partial T_1)) +\mathcal{H}^1(\bar{\phi}(\partial T_2))\leq \eps
	\end{equation}
	and
	\begin{equation}\label{GeoTry}
		\Psi_\alpha(\bar{\phi}_{\rceil \partial T_1}) + \Psi_\alpha(\bar{\phi}_{\rceil \partial \Delta}) +\Psi_\alpha(\bar{\phi}_{\rceil \partial T_2}) < \Psi_\alpha(\phi) + \eps
	\end{equation}
	hold.
\end{itemize} 
\end{itemize}
\end{lemma}

\begin{proof}
It suffices to prove the claim for the case of $\Q$ being class $ii)$ since the class $iii)$ case is just a repetition of the same argument used at a pair of opposing vertices of $\Q$.

Assume that $\Q$ is of class $ii)$, then there exists a side of $\cR_\Q$ (the smallest rectangle containing $\Q$ with sides parallel and perpendicular to $e_{\alpha}, e_{\alpha}^{\bot}$) whose intersection with $\Q$ is exactly the vertex $W$ of $\Q$. Up to a rotation of angle $\alpha$ and a translation of $(-a^-, -b^-)$ we may assume that $\cR_\Q =  [0,\ell]\times [0,h]$ for $\ell = a^+ - a^-$ and $h = b^+ - b^-$. Further we may assume that the $\alpha$ rotation of $\Q$ has a horizontal side lying in $[0,\ell]\times\{h\}$ of that $W = (w,0)$ for some $w\in [0,\ell]$. It suffices to prove our claim for the $\alpha$-rotated, translated $\Q$ and replacing $\Psi_{\alpha}$ with $\Psi_0$.

Since $\Q$ is convex, it holds that for every $t \in (0,h)$ one has $\partial \Q \cap (\R \times \{t \})$ consists of exactly two points $H_t^1$ and $H_t^2$ with $0 \leq (H_t^1)_1 < (H_t^2)_1 \leq \ell$.  Analogously, for every $s \in (0,\ell)$ one has $\partial \Q \cap (\{s\} \times \R)$ consists of exactly two points $V_s^1$ and $V_s^2$ with $0 \leq (V_s^1)_1 < (V_s^2)_1 \leq h$. 

Let $\eps >0$ be arbitrary fixed.  We then let $\bar{\delta}= \bar{\delta}(\varphi(\partial \Q)) > 0$ be the parameter introduced directly after Definition \ref{def: geodesics and modified geodesics},  and $\delta_1 < \bar{\delta}$ be the parameter introduced in Lemma \ref{lemma: 2.5 Pratelli2} for the polygon given by $\varphi(\partial \Q)$ and the number $\frac{\eps}{2(h+\ell)}$. Let
\[ 0< \eta_1 <    \min\left\lbrace  \frac{\eps}{24} ,  \frac{\delta_1}{4} \right\rbrace  .  \]

For the following recall the definition of the points $H_t^{1,2}$ in \eqref{Useful}. Since $\varphi: \Q \to \R^2$ is continuous, injective and piecewise linear,  then we can find $0 < t^* \ll h$ such that the following properties hold 
\begin{itemize}
\item[i)] $|H^1_{t^*} -W| < \eta_1$,  $|H^2_{t^*} -W| < \eta_1$; 
\item[ii)] $\varphi$ is linear on each of the segments $H^1_{t^*}W$, $H^2_{t^*}W$;
\item[iii)] $|\varphi(H^1_{t^*}) -\varphi(W)| < \eta_1$ and $|\varphi(H^2_{t^*}) -\varphi(W)| < \eta_1$.
\end{itemize}

We denote $T$ the triangle $W H^1_{t^*} H^2_{t^*}$,  and $\Delta = \Q \setminus T$.
We call $W^* = \tfrac{1}{2}(H^1_{t^*} + H^2_{t^*})$, then $W^*$ is the intersection of the segment  $H^1_{t^*}H^2_{t^*}$ and the bisector of the angle at $W$. Then claim $a)$ is immediate and claim $b)$ is immediate from the choice of $t^*$. By construction, $\Delta$ is convex and the third side of $T$ is horizontal (i.e. parallel to $\alpha$) thus proving $c)$. By the triangular inequality and by property i), we have that 
\[  \mathcal{H}^1 (\partial T) <  4\eta_1. \]

We now consider $\gamma^*$ the geodesic connecting $ \varphi(H^1_{t^*})$ and $ \varphi(H^1_{t^*})$ inside the polygon identified by $\varphi(\partial \Q)$. There are two possibilities: either $\gamma^*$ is a segment or $\gamma^*$ is a bi-linear path lying inside $\varphi(\partial \Q)$ passing through $\varphi(W)$.  In both cases, we let $\bX$ be on the internal bisector of the corner $\varphi(W)$ such that $|\bX - \varphi(W)| < 2\eta_1$ and we call $\tilde{\gamma}^*$ the path $\varphi(H^1_{t^*}) \bX \varphi(H^2_{t^*})$. Notice that, being $2\eta_1 < \delta_1$,  $\tilde{\gamma}^*$ is a $\delta_1$-modification of $\gamma^*$ in the sense of Definition \ref{def: geodesics and modified geodesics}. Moreover, $\tilde{\gamma}^*$ lies in the interior of $\varphi(\partial \Q)$ and $\mathcal{H}^1(\tilde{\gamma}^*) \leq 4\eta_1$.

We now construct the extension $\bar{\varphi} : \partial T \cup \partial \Omega \to \R^2$ that is continuous, piecewise linear and injective. We let $\bar{\varphi} = \varphi$ on $\partial \Q$, then we have claim $d)$ by the choice of $t^*$ and we only need to define $\bar{\varphi}$ on the segment $H^1_{t^*}H^2_{t^*}$. 

We set $\bar{\varphi}(W^*) := \bX$ and then we define $\bar{\varphi}:  H^1_{t^*}H^2_{t^*} \to \R^2$ as the map that is linear on $H^1_{t^*} W^*$ and $W^* H^2_{t^*}$, such that $\bar{\varphi}(H^1_{t^*}H^2_{t^*}) = \tilde{\gamma}^*$. Then we have claim $e)$. Thanks to property iii) from the choice of $t^*$ and the bound on $\mathcal{H}^1(\tilde{\gamma}^{*})$ we can compute 
\[  \mathcal{H}^1(\bar{\varphi} (\partial T))  \leq 6\eta_1.   \]

For claim $f)$ we now estimate the quantity $\Psi_0 (\bar{\varphi}_{\rceil \partial T}) + \Psi_0 (\bar{\varphi}_{\rceil \partial \Delta})$. Thanks to our choice of $\delta_1$ and \eqref{eq: stesso poligono} of Lemma \ref{lemma: 2.5 Pratelli2}, we can observe 
$$
\begin{aligned}
\rho_{\bar{\varphi}(\partial T)} \big( \bar{\varphi}(H^1_{t}),  \bar{\varphi}(H^2_{t}) \big) &\leq \rho_{\varphi(\partial \Q)} \big( \varphi(H^1_{t}),  \varphi(H^2_{t}) \big) +\frac{\eps}{2(h+\ell)}  \quad \text{ for all } t \in (0,t^*), \\
\rho_{\bar{\varphi}(\partial \Delta)} \big( \bar{\varphi}(H^1_{t}),  \bar{\varphi}(H^2_{t}) \big) &\leq \rho_{\varphi(\partial \Q)} \big( \varphi(H^1_{t}),  \varphi(H^2_{t}) \big) + \frac{\eps}{2(h+\ell)}  \quad \text{ for all } t \in (t^*,h),  \\
\rho_{\bar{\varphi}(\partial \Delta)} \big( \bar{\varphi}(V^1_{s}),  \bar{\varphi}(V^2_{s}) \big) &\leq \rho_{\varphi(\partial \Q)} \big( \varphi(V^1_{s}),  \varphi(V^2_{s}) \big) + \frac{\eps}{2(h+\ell)}  \quad \text{ for all } s \in \big(0,(H^1_{t^*})_1\big) \cup \big((H^2_{t^*})_1,\ell\big).
\end{aligned}
$$

On the other hand,  for every $s \in ((H^1_{t^*})_1, (H^2_{t^*})_1)$ we call $V^3_{s} = (s, t^*)$ and we notice that the geodesic $\gamma_{\varphi(V^1_{s})\varphi(V^2_{s})}$ must intersect $\tilde{\gamma}^*$.

Since, by construction,  $\bar{\varphi}(V^3_{s}) \in \tilde{\gamma}^*$, then the maximal distance between $\bar{\varphi}(V^3_{s})$ and $\tilde{\gamma}^* \cap\gamma_{\varphi(V^1_{s})\varphi(V^2_{s})}$ is bounded by $\mathcal{H}^1(\tilde{\gamma}^*) \leq 4\eta_1 < \delta_1$.  Then from  \eqref{eq: diversi poligon} of Lemma \ref{lemma: 2.5 Pratelli2},  we get 
\[  \rho_{\bar{\varphi}(\partial T)} \big( \bar{\varphi}(V^1_{s}), \bar{\varphi}(V^3_{s}) \big) +  \rho_{\bar{\varphi}(\partial \Delta)} \big( \bar{\varphi}(V^3_{s}), \bar{\varphi}(V^2_{s}) \big) <  \rho_{\varphi(\partial \Q)} \big( \varphi(V^1_{s}),  \varphi(V^2_{s}) \big) + \frac{\eps}{2(h+\ell)}. \]

Therefore we can compute
\begin{equation}\label{eq: stima Psi 2/3}
\begin{split}
\Psi_0 (\bar{\varphi}_{\rceil \partial T}) +& \Psi_0 (\bar{\varphi}_{\rceil \partial \Delta}) \\
= & \int_0^{t^*} \rho_{\bar{\varphi}(\partial T)} \big( \bar{\varphi}(H_t^1),  \bar{\varphi}(H_t^2) \big) dt + \int_{(H^1_{t^*})_1}^{(H^2_{t^*})_1} \rho_{\bar{\varphi}(\partial T)} \big( \bar{\varphi}(V_s^1),  \bar{\varphi}(V_s^3) \big) ds \\
&+ \int_{t^*}^{h} \rho_{\bar{\varphi}(\partial \Delta)} \big( \bar{\varphi}(H_t^1),  \bar{\varphi}(H_t^2) \big) dt +  \int_{(H^1_{t^*})_1}^{(H^2_{t^*})_1} \rho_{\bar{\varphi}(\partial \Delta)} \big( \bar{\varphi}(V_s^3), \bar{\varphi}(V_s^2) \big) ds \\
&+ \int_{(0,(H^1_{t^*})_1) \cup ((H^2_{t^*})_1,\ell)} \rho_{\bar{\varphi}(\partial \Delta)} \big( \bar{\varphi}(V_s^1),  \bar{\varphi}(V_s^2) \big) ds\\
\leq & \int_0^h  \rho_{\varphi(\partial \Q)} \big( \varphi(V_t^1),  \varphi(V_t^2) \big) dt  + \int_0^{\ell} \rho_{\varphi(\partial \Q)} \big( \varphi(V_s^1),  \varphi(V_s^2) \big) ds \\
& +(\ell + h) \frac{\eps}{2(h+\ell)} \\
\leq & \Psi_0(\varphi) +  \frac{\eps}{2}.
\end{split}
\end{equation}

To finish the proof it suffices prove the claim in the case that $\Q$ is class $iii)$. But this is just a question of repeating the argument above for the opposing vertex, since the horizontal side of $\Q$ was not used at any point. Finally, in every case $\partial\Delta$ has two sides parallel to $\alpha$.
\end{proof}

The next result that we present concerns the extension of the boundary values inside a convex polygon $\Delta$ having two non-consecutive parallel sides.  This is an opportune generalization of the analogous results on rectangles proved in Lemma~2.11 of \cite{PR2}. 
Loosely speaking, there are two differences between the current setting and the one considered in Lemma 2.11 of \cite{PR2}. In \cite{PR2} the rectangular domain is partitioned in rectangular strips while here the polygon $\Delta$ is partitioned into strips that are not necessarily rectangular, which we later split further into a rectangle and two triangles, one at either end.

\begin{lemma}[Skeleton-strips]\label{lemma: skeleton strips}
Let $h,\ell > 0$ and let $\Delta \subset [0,\ell] \times [0,h]$ be a convex polygon and let $[0,\ell] \times [0,h]$ be the smallest rectangle containing $\Delta$. Further assume that $\partial \Delta$ has a pair of horizontal sides, one of which lies in $[0,\ell] \times \{0\}$ and the other in $[0,\ell] \times \{h\}$. For every $\phi : \partial \Delta \to \R^2$ continuous piecewise linear injective map and for every $\eps >0$ there exist $M\in \N$ and values 
\[ 0 = t_0 < t_1 < \dots < t_{M-1} < t_M = h  \]
such that the following properties hold: 
\begin{itemize}
\item[i)] $t_{i+1} - t_i < \eps$ for every $i = 0, \dots, M-1$

\item[ii)] for every $i = 0, \dots, M-1$, call $S_i := \Delta \cap \big( \R \times [t_i,t_{i+1}] \big)$ (which we call a horizontal strip). Then $\phi$ is linear on $I_i^1 $ and $I_i^2$, where $I_i^1, \,I_i^2$ are the two non horizontal segments of $\partial S_i \cap \partial \Delta$.  Moreover, 
\begin{equation}\label{eq: estimate boundary segments}
\mathcal{H}^1 (\phi(I_i^1)) + \mathcal{H}^1 (\phi(I_i^2)) < \eps
\end{equation}

\item[iii)] there exists $\bar{\phi} : \bigcup_{i=0}^{M-1} \partial S_i \to \R^2$ a continuous, piecewise linear, injective map such that $\bar{\phi} = \phi$ on $\partial \Delta$ and 
\begin{equation}\label{eq: estimate psi functions on strips}
\sum_{i=0}^{M-1} \Psi_0 (\bar{\phi} _{\rceil \partial S_i}) \leq \Psi_0 (\phi) + \eps 
\end{equation}

\item[iv)] for every $i = 0, \dots, M-1$,  the quadrilateral $S_i$ can be decomposed in the essentially disjoint union $T_i^1 \cup R_i \cup T_i^2$, where $R_i$ is a rectangle with horizontal and vertical sides and $T_i^1,T_i^2$ are right angle triangles whose hypotenuses are $I_i^{1},  I_i^{2}$ the two segments of $\partial \Delta \cap \partial S_i$ 

\item[v)] there exists $\tilde{\phi} : \bigcup_{i=0}^{M-1} \partial T_i^1 \cup \partial R_i \cup  \partial T_i^2  \to \R^2$ a continuous, piecewise linear, injective map such that $\tilde{\phi} = \bar{\phi}$ on $\bigcup_{i=0}^{M-1} \partial S_i$ and moreover
\begin{equation}\label{eq: estimate psi functions on skeleton}
\sum_{i=0}^{M-1} \big(  \Psi_0 (\tilde{\phi} _{\rceil \partial T_i^1}) +  \Psi_0 (\tilde{\phi} _{\rceil \partial R_i}) +  \Psi_0 (\tilde{\phi} _{\rceil \partial T_i^2}) \big) \leq \Psi_0 (\phi) + \eps.
\end{equation}
\end{itemize}
\end{lemma}

\begin{figure}
\begin{tikzpicture}[line cap=round,line join=round,>=triangle 45,x=0.37cm,y=0.37cm]
\clip(-0.5,-1) rectangle (38.5,10.);
\fill[line width=0.5pt,color=qqwuqq,fill=qqwuqq,fill opacity=0.10000000149011612] (3,1.) -- (11,1.) -- (15.,9.) -- (1.,9.) -- cycle;
\fill[line width=0.5pt,color=qqttzz,fill=qqttzz,fill opacity=0.10000000149011612] (19.30505826175965,0.7888168118674314) -- (25.62448447945863,5.417128971308936) -- (24.33389743499898,0.9668288179997969) -- (32.56695271862089,1.1448408241321626) -- (30.386305643499412,3.904026919183829) -- (34.,6.) -- (36.30520484740057,1.5008648363968937) -- (38.30783991638968,5.417128971308936) -- (34.61409078914309,9.066375097022432) -- (27.938640559179383,4.660577945246383) -- (27.27109553618301,8.754854086290791) -- (23.66635241200261,9.422399109287163) -- (19.26055526022656,7.553273044897324) -- (24.28939443346589,6.7522190173016785) -- (20.061609287822204,5.6396439789743935) -- cycle;
\draw [line width=0.5pt,color=qqwuqq] (3,1.)-- (11,1.);
\draw [line width=1.5pt,color=qqwuqq] (3,1.)-- (2.5,3.);
\draw [line width=1.5pt,color=qqwuqq] (11,1.)-- (12,3);
\draw [line width=0.5pt,color=qqwuqq] (11,1.)-- (15.,9.);
\draw [line width=0.5pt,color=qqwuqq] (15.,9.)-- (1.,9.);
\draw [line width=0.5pt,color=qqwuqq] (1.,9.)-- (3,1.);
\draw [line width=0.5pt] (2.5,3)-- (12,3);
\draw [line width=0.5pt,color=qqttzz] (19.30505826175965,0.7888168118674314)-- (25.62448447945863,5.417128971308936);
\draw [line width=0.5pt,color=qqttzz] (25.62448447945863,5.417128971308936)-- (24.33389743499898,0.9668288179997969);
\draw [line width=0.5pt,color=qqttzz] (24.33389743499898,0.9668288179997969)-- (32.56695271862089,1.1448408241321626);
\draw [line width=0.5pt,color=qqttzz] (32.56695271862089,1.1448408241321626)-- (30.386305643499412,3.904026919183829);
\draw [line width=0.5pt,color=qqttzz] (30.386305643499412,3.904026919183829)-- (34.,6.);
\draw [line width=0.5pt,color=qqttzz] (34.,6.)-- (36.30520484740057,1.5008648363968937);
\draw [line width=0.5pt,color=qqttzz] (36.30520484740057,1.5008648363968937)-- (38.30783991638968,5.417128971308936);
\draw [line width=0.5pt,color=qqttzz] (38.30783991638968,5.417128971308936)-- (34.61409078914309,9.066375097022432);
\draw [line width=0.5pt,color=qqttzz] (34.61409078914309,9.066375097022432)-- (27.938640559179383,4.660577945246383);
\draw [line width=0.5pt,color=qqttzz] (27.938640559179383,4.660577945246383)-- (27.27109553618301,8.754854086290791);
\draw [line width=0.5pt,color=qqttzz] (27.27109553618301,8.754854086290791)-- (23.66635241200261,9.422399109287163);
\draw [line width=0.5pt,color=qqttzz] (23.66635241200261,9.422399109287163)-- (19.26055526022656,7.553273044897324);
\draw [line width=0.5pt,color=qqttzz] (19.26055526022656,7.553273044897324)-- (24.28939443346589,6.7522190173016785);
\draw [line width=0.5pt,color=qqttzz] (24.28939443346589,6.7522190173016785)-- (20.061609287822204,5.6396439789743935);
\draw [line width=0.5pt,color=qqttzz] (20.061609287822204,5.6396439789743935)-- (19.30505826175965,0.7888168118674314);
\draw [line width=0.5pt,color=ffqqqq] (25.438946697626786,0.9907217750295874)-- (25.331043767514927,1.033179502920421);
\draw [line width=0.5pt,color=ffqqqq] (25.331043767514927,1.033179502920421)-- (24.81309506138824,1.1680750181803181);
\draw [line width=0.5pt,color=ffqqqq] (24.81309506138824,1.1680750181803181)-- (26.078602946416105,6.185546421494993);
\draw [line width=0.5pt,color=ffqqqq] (26.078602946416105,6.185546421494993)-- (19.58261848345801,1.2927311342413874);
\draw [line width=0.5pt,color=ffqqqq] (19.58261848345801,1.2927311342413874)-- (19.339386253463527,1.008919817498163);
\draw [->,line width=0.5pt] (23,2) -- (23,3.7);
\draw [line width=0.5pt,dotted] (2.5,3.) -- (2.5,-1.);
\draw [line width=0.5pt,dotted] (12.,3.) -- (12.,-1.);
\draw [line width=0.5pt,dotted] (11.,1.) -- (11.,-1.);
\draw [line width=0.5pt,dotted] (3,1.) -- (3.,-1.);
\draw (-0.5, 0.5) node[anchor=north west] {$d_2, c_2$};
\draw (2.7,0.2) node[anchor=north west] {$x_1$};
\draw (9.6,0.2) node[anchor=north west] {$x_2$};
\draw (11.7,0.5) node[anchor=north west] {$c_2, d_2$};
\draw (6,6) node[anchor=north west] {$\Delta^+$};
\draw (6,2.7) node[anchor=north west] {$S_0$};
\draw (27.5,8) node[anchor=north west] {$\P$};
\draw (28,3) node[anchor=north west] {$\P^+$};
\draw (22,2.2) node[anchor=north west] {$\P_0$};
\draw (23.1,6) node[anchor=north west] {{\color{red}$\bar{\gamma}_1$}};
\draw (12,4) node[anchor=north west] {$H_{t_1}^2$};
\draw (0.1,4) node[anchor=north west] {$H_{t_1}^1$};
\draw (17.4,2) node[anchor=north west] {$\bH_{t_1}^1$};
\draw (25,1.3) node[anchor=north west] {$\bH_{t_1}^2$};
\draw (2.7,2.8) node[anchor=north west] {$I_0^{1}$};
\draw (9.8,2.8) node[anchor=north west] {$I_0^{2}$};
\begin{scriptsize}
\draw [fill=black] (2.5,3) circle (2pt);
\draw [fill=black] (12,3) circle (2pt);
\draw [fill=black] (19.339386253463527,1.008919817498163) circle (2pt);
\draw [fill=black] (25.438946697626786,0.9907217750295874) circle (2pt);
\end{scriptsize}
\end{tikzpicture}
\caption{The figure shows the slicing of the set $\Delta $ into $\Delta^+\cup S_0$ by the horizontal line $\R\times\{t_1\}$ and $\P$ into $\P^+ \cup \P_0$ by the modified geodesic called $\bar{\gamma}_1$.}
\label{Fig:Slicing}
\end{figure}
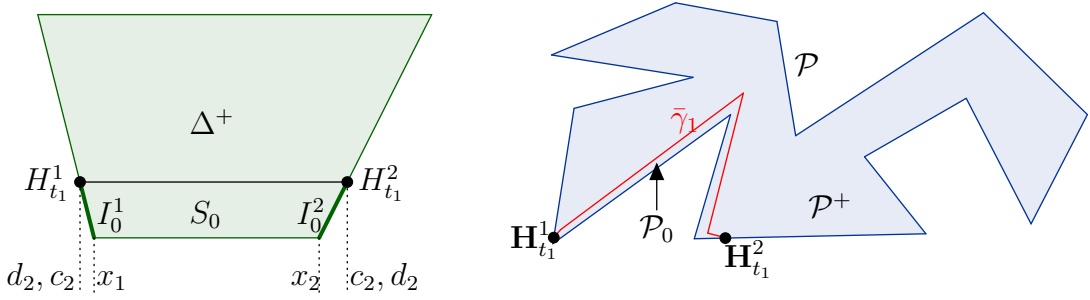

\begin{proof}
Let $\eps >0$ be fixed.  Throughout the proof we denote by $\P$ the polygon of boundary $\varphi(\partial \Delta)$.

\step{I}{Finding the value $M$ and fixing $(t_i)_{i=0 \ldots M}$}{Skel1}

There are a finite number of vertices of $\Delta$ call it $M_1$. There is similarly a finite number of vertices of $\P$, call it $M_2$. Since $|D_{\tau}\phi|\in L^{\infty}(\partial \Delta)$ we have that any segment on $\partial \Delta$ of length at most $\epsilon(1+\|D_{\tau}\phi\|_{\infty})^{-1}$ has image whose length is at most $\epsilon$. Further since $\H^1(\partial \Delta)<\infty$, we find a number $M_3$ bounded by $\epsilon^{-1}\H^1(\partial \Delta)(1+|D_{\tau}\phi|)$ such that splitting $\partial\Delta$ into $M_3$ segments then the length of the segments and their images is bounded by $\epsilon$.

We take set of all $\tilde{t}\in (0,h)$ the $t$-coordinates such that $\Delta$ has a vertex with $t$-coordinate equal to $\tilde{t}$ (their number is bounded by $M_1$) further all $\tilde{t}\in (0,h)$ such that $\partial_{\tau}\phi$ does not exist (their number is bounded by $M_2$) and then add a finite number (bounded by $M_3$) of $\tilde{t}$ such that whenever $\tilde{t}, \tilde{t}^*$ are a pair of neighbouring (with respect to the order $<$) we have
$$
|H_{\tilde{t}}^1 - H_{\tilde{t}^*}^1|<\epsilon> |H_{\tilde{t}}^2 - H_{\tilde{t}^*}^2| \quad \text{ and } \quad |\bH_{\tilde{t}}^1 - \bH_{\tilde{t}^*}^1|<\epsilon> |\bH_{\tilde{t}}^2 - \bH_{\tilde{t}^*}^2|.
$$  
Indexing this set from $\{1, \dots, M-1\}$ and calling $t_0 =0$ and $t_M = h$ we determine $t_i$ and the number $M$.

We define the strips $S_i  = (\er\times[t_i,t_{i+1}]) \cap \Delta$. They are all convex quadrilaterals with two horizontal sides.

\smallskip

\step{II}{Definition of the curve $\bar{\gamma}_1 = \bar{\phi}( \Delta \cap (\R \times \{t_1 \}))$ and the polygons $\Delta^+\cup S_0 = \Delta$ and $\P^+ \cup \P_0 =  \P$ }{Skel2}

The goal of this step is to define the piecewise linear curve $\bar{\gamma}_1$,  internal to $\P$, which will be the image of the segment $H_{t_1}^1H_{t_1}^2$ in a map $\bar{\phi}$ extending $\varphi$. The precise parameterization of $\bar{\phi}$ will be presented in the next step, here we only aim to define the curve $\bar{\gamma}_1 \subset \P$. 

Our argument is recursive and so we deal with the first curve $\bar{\gamma}_1$ defined on $H_{t_1}^1H_{t_1}^2$ separating $\Delta$ into $S_0$ and $\Delta\cap \er\times[t_1, h] = \Delta^+$ (see Figure~\ref{Fig:Slicing}).  Similarly the curve $\bar{\gamma}_1$ means dividing the polygon $\P$ into two further polygons: a polygon $\P_0$ (which will be the image of $S_0$) containing the curve $\varphi(H_{ t_0}^1H_{t_0}^2)$ and another polygon $\P^+$ (which will be the image of $\Delta^+$) (see Figure~\ref{Fig:Slicing}).

Since $\P$ is a non degenerate polygon,  let $\bar{\delta}(\P) > 0$ be the parameter of Definition \ref{def: geodesics and modified geodesics} and let $\delta_1 > 0$ be so small that 
\[  \delta_1  < \min \left\lbrace  \bar{\delta}(\P), \frac{\eps (t_1 - t_0)}{8h\mathcal{H}^1(\P)},  \frac{h}{2^3},  \frac{\eps}{2}   \right\rbrace \ \text{and Lemma \ref{lemma: 2.5 Pratelli2} applies with $\delta_1$ for $\P$ and }  \frac{\eps (t_1 - t_0)}{8h(\ell + h)}.   \]

We define
\begin{equation}\label{eq: bargamma1 is modification with variable endpoints}
\bar{\gamma}_1 \text{as a $\delta_1$-modification of the geodesic in $\P$ connecting $\bH^1_{t_1}$ and $\bH^2_{t_1}$}.
\end{equation}

\smallskip

\step{III}{Definition of $\bar{\phi}$ on $\partial S_0$}{Skel3}

In this step we care about the definition of $\bar{\phi}$ on $\partial S_0$.  More precisely,  we let $\bar{\phi} = \varphi$ on $\partial \Delta$ and we specify the parametrization $\bar{\phi}:  \partial \Delta^+ \cap (\R \times \{ t_1 \})  \to  \bar{\gamma}_1$ so that $\bar{\phi}$ is continuous, injective, piecewise linear and 
\begin{equation}\label{eq: estimate Psi function on S0}
\Psi_0(\bar{\phi} _{\rceil \partial S_0}) + \Psi_0(\bar{\phi} _{\rceil \partial \Delta^+}) \leq \Psi_0(\varphi) + \frac{\eps}{2h} (t_1 - t_0). 
\end{equation}

Let us observe, that thanks to Lemma \ref{lemma: 3.7 Pratelli3} it is enough to look for a continuous and injective parametrization $\psi: \partial \Delta^+ \cup \partial S_0  \to  \R^2$ coinciding with $\varphi$ on $\partial \Delta$ such that \eqref{eq: estimate Psi function on S0} holds for $\psi$ with error $\frac{\eps}{4h}(t_1 -t_0)$, namely 
\[  \Psi_0(\psi _{\rceil \partial S_0}) + \Psi_0(\psi _{\rceil \partial \Delta^+}) < \Psi_0(\varphi) + \frac{\eps}{4h} (t_1 - t_0).   \]
Indeed,  the correct $\bar{\phi}$ can be found as a $\delta$-linearization of $\psi$ for some $\delta$ small enough depending on $\frac{\eps}{4h}(t_1 - t_0)$ such that 
\[  \Psi_0(\bar{\phi} _{\rceil \partial S_0}) + \Psi_0(\bar{\phi} _{\rceil \partial \Delta^+}) <   \Psi_0(\psi _{\rceil \partial S_0}) + \Psi_0(\psi _{\rceil \partial \Delta^+}) + \frac{\eps}{4h}(t_1 - t_0).   \]

Thanks to our choice of $\delta_1$ and the fact that $\bar{\gamma}_1$ is a $\delta_1$-modification with variable endpoints of the geodesic connecting $\bH^1_{t_1}$ and $\bH^2_{t_1}$,  hence splitting $\P$ into the two polygons $\P_0$ and $\P^+$, we can apply Lemma \ref{lemma: 2.5 Pratelli2} to get that 
\begin{equation}\label{eq: tutte le t}
\begin{aligned}
\rho_{\P_0} \big( \bH^1_t,  \bH^2_t  \big) &\leq \rho_{\P} \big( \bH^1_t,  \bH^2_t  \big) + \frac{\eps (t_1 - t_0)}{8h(\ell + h)} \ \text{for any $t_0<t<t_1$ and } \\
  \rho_{\P^+} \big( \bH^1_t,  \bH^2_t  \big) &\leq \rho_{\P} \big( \bH^1_t,  \bH^2_t  \big) + \frac{\eps (t_1 - t_0)}{8h(\ell + h)} \ \text{for any $t_1<t<t_M$}.
\end{aligned}
\end{equation}

For short denote $c_1 = (H_{t_1}^1)_1$ and $c_2 = (H_{t_1}^2)_1$. Then $0 \leq c_1 < c_2 \leq \ell$. For every $0 < s < \ell$ we call $\gamma_s$ the geodesic inside $\P$ connecting $\bV_s^1$ and $\bV_s^2$. Moreover, whenever $c_1 < s < c_2$ we also set $V_s^3 := (s,t_1)$ the point in the intersection of $H_{t_1}^1H_{t_1}^2$ with $V_{s}^1V_{s}^2$.  
For every $s \in (0,c_1) \cup (c_2,\ell)$ we have that either both $V_{s}^1,V_{s}^2 \in S_0$ and using Lemma \ref{lemma: 2.5 Pratelli2}
$$
\rho_{\P_0} \big( \bV_{s}^1,  \bV_{s}^2 \big) \leq \rho_{\P} \big( \bV_{s}^1,  \bV_{s}^2 \big) + \frac{\eps (t_1 - t_0)}{8h(\ell + h)}
$$
or $V_{s}^1,V_{s}^2 \in \Delta^+$ and by Lemma \ref{lemma: 2.5 Pratelli2}
$$
\rho_{\P^+} \big( \bV_{s}^1,  \bV_{s}^2 \big) \leq \rho_{\P} \big( \bV_{s}^1,  \bV_{s}^2 \big) + \frac{\eps (t_1 - t_0)}{8h(\ell + h)}
$$
The two equations above can be expressed simultaneously as 
\begin{equation}\label{eq: s esterne}
\max\big\{\rho_{\P_0} \big( \bV_{s}^1,  \bV_{s}^2 \big), \rho_{\P^+} \big( \bV_{s}^1,  \bV_{s}^2 \big)\big\} \leq \rho_{\P} \big( \bV_{s}^1,  \bV_{s}^2 \big) + \frac{\eps (t_1 - t_0)}{8h(\ell + h)}
\end{equation}
for all $s \in (0,c_1) \cup (c_2,\ell)$.

On the other hand, whenever $s \in (c_1,c_2)$, the points $\bV_s^1 \in \P_0$ and $\bV_s^2 \in \P^+$ thus the geodesic $\gamma_s$ necessarily intersects $\bar{\gamma}_1$. Let $\kappa$ be the (injective and continuous) constant-speed parametrization of $\bar{\gamma}_1$ from $[0,\mathcal{H}^1(\bar{\gamma}_1)]$, $\kappa(0) = \bH^1_{t_1}$ and $\kappa(\mathcal{H}^1(\bar{\gamma}_1)) = \bH^2_{t_1}$. For every $s \in (c_1, c_2)$ we then let $\bX(s)$ be the point in $\gamma_s \cap \bar{\gamma}_1$ such that
$$
	\bX(s) = \kappa\Big(\max \Big\{x\in [0,\mathcal{H}^1(\bar{\gamma}_1)]: \kappa(x) \in \gamma_s \cap \bar{\gamma}_1\Big\}\Big).
$$

Then,  thanks to Lemma \ref{lemma: 2.4 Pratelli2}, it is easy to see that the map $s \to \kappa^{-1}(\bX(s))$ is non-decreasing, therefore if $c_1 < s < s' < c_2$ then $\bX(s') \in \kappa([\kappa^{-1}(\bX(s)), \mathcal{H}^1(\bar{\gamma}_1)])$. Notice that, in general, the function $s\to \kappa^{-1}\circ\bX(s)$ is not continuous, nor injective nor surjective. However, for every $\sigma>0$ it is always possible to find a continuous bijection $\bX_\sigma$ of $[c_1,c_2]$ onto $\bar{\gamma}_1$ such that 
\begin{equation}\label{eq: punti che scazzano nella riparametrizzazione}
\mathcal{H}^1 \big( J_\sigma \big) < \sigma  \quad \text{ where } \quad  J_\sigma:= \{  s \in (c_1, c_2) : | \bX_\sigma(s) - \bX(s) | > \sigma \}  . 
\end{equation}
We can then fix $\sigma = \frac{\delta_1}{2}$ and apply Lemma \ref{lemma: 2.5 Pratelli2} to get that 
\begin{equation}\label{eq: s belle interne}
\rho_{\P_0} \big( \bV_s^1,  \bX_{\frac{\delta_1}{2}}(s) \big) + \rho_{\P^+} \big( \bX_{\frac{\delta_1}{2}}(s) ,  \bV_s^2 \big)   \leq \rho_{\P} \big( \bV_s^1,  \bV_s^2 \big) + \frac{\eps (t_1 - t_0)}{8h(\ell + h)}
\end{equation}
for all $s \in (c_1,c_2) \setminus J_{\frac{\delta_1}{2}}$. On the other hand,  we have the trivial estimate 
\begin{equation}\label{eq: s brutte interne}
\rho_{\P_0} \big( \bV_s^1,  \bX_{\frac{\delta_1}{2}}(s) \big) + \rho_{\P^+} \big( \bX_{\frac{\delta_1}{2}}(s) ,  \bV^2_s \big)   \leq \mathcal{H}^1(\partial \P) + \mathcal{H}^1(\bar{\gamma}_1) < 2 \mathcal{H}^1(\partial \P) 
\end{equation}
 for all $s \in J_{\frac{\delta_1}{2}}$.
 
We define $\psi: \partial \Delta^+ \cup \partial S_0 \to \R^2$ as $\psi = \varphi$ on $\partial \Delta$ and $\psi(V_s^3) = \bX_{\frac{\delta_1}{2}} (s)$ for every $s \in (c_1, c_2)$. In particular, $\psi$ is continuous and injective and fails to be piecewise linear only on the segment $H^1_{t_1}H^2_{t_1}$. Moreover, gathering together \eqref{eq: tutte le t}, \eqref{eq: s esterne}, \eqref{eq: s belle interne}, \eqref{eq: s brutte interne} we deduce that 
$$
	\begin{aligned}
		\Psi_0 (\psi _{\rceil \partial S_0}) + \Psi_0 (\psi _{\rceil \partial \Delta^+}) =& \int_0^{t_1} \rho_{\P_0} \big( \psi(H^1_{t}), \psi(H^2_{t}) \big) dt + \int_{t_1}^h \rho_{\P^+} \big( \psi(H^1_{t}), \psi(H^2_{t}) \big) dt \\
		& + \int_{(0,c_1) \cup (c_2,\ell)} \min\big\{\rho_{\P_0} \big( \bV_{s}^1,  \bV_{s}^2 \big) + \rho_{\P^+} \big( \bV_{s}^1,  \bV_{s}^2 \big)\big\}  \\
		& + \int_{(c_1,c_2) \setminus J_{\frac{\delta_1}{2}}}  \rho_{\P_0} \big( \psi(V_s^1), \psi(V_s^3) \big) + \rho_{\P^+} \big( \psi(V_s^3),   \psi(V_s^2) \big)  \\
		& + \int_{J_{\frac{\delta_1}{2}}}  \rho_{\P_0} \big( \psi(V_s^2),   \psi(V_s^3) \big) + \rho_{\P^+} \big( \psi(V_s^3),   \psi(V_s^2) \big) \\
		\leq & \int_0^h \rho_{\P} \big( \varphi(H_t^1),  \varphi(H_t^2) \big) dt +   \int_{(0,\ell) \setminus J_{\frac{\delta_1}{2}}} \rho_{\P} \big( \varphi(V_s^1),  \varphi(V_s^2) \big) ds \\
		& +  \frac{\eps (t_1 - t_0)}{8h(\ell + h)} (\ell + h - \mathcal{H}^1(J_{\frac{\delta_1}{2}})) + \mathcal{H}^1(J_{\frac{\delta_1}{2}}) 2 \mathcal{H}^1(\partial \P)  \\
		\leq & \Psi_0(\varphi) + \frac{\eps}{4h}(t_1 - t_0),
	\end{aligned}
$$
where in the last inequality we used \eqref{eq: punti che scazzano nella riparametrizzazione} and the fact that $\delta_1 < \frac{\eps (t_1 - t_0)}{8h\mathcal{H}^1(\P)}$. 

Finally,  thanks to Lemma \ref{lemma: 3.7 Pratelli3} and the considerations of the first part of the step,  we can find a function $\bar{\phi} : \partial \Delta^+ \cup \partial S_0 \to \R^2$ that is continuous, injective, piecewise linear and such that \eqref{eq: estimate Psi function on S0} holds.

\smallskip
\step{IV}{Definition of $\tilde{\varphi}$ on $\partial T_0^1 \cup \partial R_0 \cup \partial T_0^2$}{Skel4}

In this step we further subdivide the strip $S_0$ in the essentially disjoint union of two triangles $T_0^1,T_0^2$ and a rectangle $R_0$ with the following properties. The rectangle $R_0$ is the biggest rectangle with horizontal and vertical sides inside $S_0$, such that the horizontal sides are contained in $\partial S_0$, while $T_0^1$, $T_0^2$ are the two disjoint right-angle triangles containing $I_0^{1}, I_0^{2}$ respectively.  


We continue to define a new $\tilde{\varphi} : \partial \Delta^+ \cup \partial T_0^1 \cup \partial R_0 \cup \partial T_0^2 \to \R^2$ coinciding with $\bar{\phi}$ on $ \partial \Delta^+ \cup \partial S_0$ such that $\tilde{\varphi}$ is injective, continuous, piecewise linear and satisfies the following estimate 
\begin{equation}\label{eq: estimate Psi function on T0l R0 T0r}
\Psi_0(\tilde{\phi} _{\rceil \partial T_0^1}) + \Psi_0(\tilde{\phi} _{\rceil \partial R_0}) + \Psi_0(\tilde{\phi} _{\rceil \partial T_0^2}) \leq \Psi_0(\bar{\phi} _{\rceil \partial S_0}) + \frac{\eps}{2h} (t_1 - t_0).
\end{equation}
Let us emphasize that $\tilde{\phi}$ will be defined so that $\tilde{\phi} (\partial T_0^1 \cup \partial R_0 \cup \partial T_0^2) \subset \P_0$.

We denote by $d_1$ and $d_2$ the two values such that the projection of $\partial S_0$ onto $\R \times \{0\}$ is exactly $[d_1,d_2]\times \{ 0\}$, and we call $x_1$ and $x_2$ those values for which the projection of $\partial R_0$ onto $\R \times \{0\}$ is the segment $[x_1, x_2]\times\{0\}$.
Notice that $d_1\leq c_1 \leq x_1 < x_2 \leq c_2 \leq d_2$.  

Being $\P_0$ a non degenerate polygon,  we let $\bar{\delta}(\P_0)$ be the parameter of Definition \ref{def: geodesics and modified geodesics} and take 
\[  \delta'_1 < \min \left\lbrace \bar{\delta}(\P_0),  \frac{\eps(t_1 - t_0)}{32h\mathcal{H}^1(\P_0)}  \right\rbrace \ \text{and Lemma \ref{lemma: 2.5 Pratelli2} applies with $\delta'_1$ for $\P_0$ and }  \frac{\eps (t_1 - t_0)}{16h(\ell + h)}. \]

Let now $\nu_{x_1}$ be the geodesic inside $\P_0$ connecting $\bar{\phi}(V_{x_1}^1)$ and $\bar{\phi}(V^3_{x_1})$ and $\bar{\nu}_{x_1}$ be its $ \delta'_1$-modification in the sense of Definition \ref{def: geodesics and modified geodesics}. In particular,  $\bar{\nu}_{x_1}$ splits $\P_0$ into two non degenerate polygons $\P_0^1$ and $\U$,  where $\P_0^1$ contains $\bar{\phi}(I_0^1)$ and $\U$ contains $\bar{\phi}(I_0^2)$. 

Thanks to Lemma \ref{lemma: 2.5 Pratelli2} we have
\begin{equation}\label{eq: tutte le s in a- a+}
\begin{split}
\rho_{\U} \big( \bar{\phi}(V_s^1),  \bar{\phi}(V_s^3) \big) &\leq \rho_{\P_0}  \big( \bar{\phi}(V^1_s),  \bar{\phi}(V_s^3) \big) + \frac{\eps (t_1 - t_0)}{16h(\ell + h)}, \text{ for all $s \in (x_1, x_2)$}\\
\rho_{\U} \big( \bar{\phi}(V_s^1),  \bar{\phi}(V_s^2) \big) &\leq \rho_{\P_0}  \big( \bar{\phi}(V_s^1),  \bar{\phi}(V_s^2) \big) + \frac{\eps (t_1 - t_0)}{16h(\ell + h)}, \text{ for all $s \in (x_2, d_2)$} 
\end{split}
\end{equation}
while for all $s \in (d_1,x_1)$
\begin{equation}\label{eq: tutte le s fuori da a- a+}
\rho_{\P_0^1} \big( \bar{\phi}(V_s^1),  \bar{\phi}(V_s^2) \big)   \leq \rho_{\P_0}  \big( \bar{\phi}(V_s^1),  \bar{\phi}(V_s^2) \big) + \frac{\eps (t_1 - t_0)}{16h(\ell + h)}. 
\end{equation}

We continue similarly as described in step~\ref{Skel3}. For every $t \in (0,t_1)$ we denote the point $\bar{\phi}(H_t^1) = \bH_t^1 \in \P_0^1$ and $\bar{\phi}(H_t^2) = \bH_t^2 \in \U$ thus the geodesic $\zeta_t$ connecting $\bH_t^1$ and $\bH_t^2$ inside $\P_0$ must intersect $\bar{\nu}_{x_1}$.  So also in this case, for every $t \in (0,t_1)$ we can find a map $\bY(t)$ identifying the last point of the intersection $\zeta_t \cap \bar{\nu}_{x_1}$ running $\bar{\nu}_{x_1}$ from $\bar{\phi}(V_{x_1}^1)$ to $\bar{\phi}(V_{x_1}^3)$.  Moreover, exactly as explained in step~\ref{Skel3}, we can find a continuous and injective approximation $\bY_{\frac{\delta'_1}{2}}$ such that 
$$
	| \bY(t) - \bY_{\frac{\delta'_1}{2}}(t)| < \frac{\delta'_1}{2} \ \text{ for all $t \in (0,t_1) \setminus J^1_{\frac{\delta'_1}{2}} \ $    and } \ \mathcal{H}^1\big(J^1_{\frac{\delta'_1}{2}} \big)  <  \frac{\delta'_1}{2}.
$$
So if we now call $H_t^3 := (t,x_1)$ then we can define $\psi^1 : \partial T_0^1 \cup \partial S_0 \to \R^2$ in this way: $\psi^1 = \bar{\phi}$ on $\partial S_0$ and $ \psi^1(H_t^3) = \bY_{\frac{\delta'_1}{2}}(t)$ for all $t \in (0,t_1)$. Then the map $\psi^1$ is continuous, injective and fails to be piecewise linear only on the segment $\{x_1\} \times [0,t_1]$.  Using Lemma \ref{lemma: 2.5 Pratelli2},  for all $t \in(0,t_1) \setminus J^1_{\frac{\delta'_1}{2}}$ we can estimate 
\begin{equation}\label{eq: t buone}
	\begin{aligned}
		\rho_{\P_0^1} \big(\psi^1(H_t^1),  \psi^1(H_t^3) \big) &+ \rho_{\U} \big(\psi^1(H_t^3),  \psi^1(H_t^2) \big) \\
		&\leq \rho_{\P_0} \big( \bar{\phi}(H_t^1),  \bar{\phi}(H_t^2) \big) +  \frac{\eps (t_1 - t_0)}{16h(\ell + h)}
	\end{aligned}
\end{equation}
while for the remaining $t \in J^l_{\frac{\delta'_1}{2}}$ we get 
\begin{equation}\label{eq: t cattive}
	\begin{aligned}
		\rho_{\P_0^1} \big(\psi^1(H_t^1),  \psi^1(H_t^3) \big) + \rho_{\U} \big(\psi^1(H_t^3),  \psi^1(H_t^2), \big)  &\leq \mathcal{H}^1(\bar{\nu}_{x_1}) + \mathcal{H}^1(\partial \P_0)\\ 
		&\leq 2 \mathcal{H}^1(\partial \P_0). 
	\end{aligned}
\end{equation}

Then \eqref{eq: tutte le s in a- a+}, \eqref{eq: tutte le s fuori da a- a+}, \eqref{eq: t buone} and \eqref{eq: t cattive} give 
\begin{align*}
\Psi_0(\psi^1 _{\rceil \partial T^1_0}) +& \Psi_0(\psi^1 _{\rceil \partial (S_0 \setminus T^1_0)})\\
 \leq & \int_{d_1}^{c_1} \rho_{\P_0}  \big( \bar{\phi}(V_s^1),  \bar{\phi}(V_s^2) \big)  + \frac{\eps (t_1 - t_0)}{16h(\ell + h)} ds\\
 &+ \int_{c_1}^{c_2} \rho_{\P_0}  \big( \bar{\phi}(V_s^1),  \bar{\phi}(V_s^3) \big)    + \frac{\eps (t_1 - t_0)}{16h(\ell + h)} ds  \\
 &+\int_{c_2}^{d_2} \rho_{\P_0}  \big( \bar{\phi}(V_s^1),  \bar{\phi}(V_s^2) \big) +  \frac{\eps (t_1 - t_0)}{16h(\ell + h)} ds\\
& +\int_{(0,t_1) \setminus J^1_{\frac{\delta'_1}{2}}} \rho_{\P_0} \big( \bar{\phi}(H_t^1),  \bar{\phi}(H_t^2) \big) +  \frac{\eps (t_1 - t_0)}{16h(\ell + h)} dt \\
&+ 2 \mathcal{H}^1\big( J^1_{\frac{\delta'_1}{2}} \big)\mathcal{H}^1(\partial \P_0) \\
\leq & \Psi_0(\bar{\phi} _{\rceil \P_0}) + \frac{\eps}{8h}(t_1 - t_0),
\end{align*}
and, as in step~\ref{Skel3},  Lemma \ref{lemma: 3.7 Pratelli3} ensures that there is some continuous, injective and piecewise linear map $\tilde{\varphi}^1 : \partial T_0^1\cup \partial S_0 \to \R^2$ such that $\tilde{\varphi}^1 = \psi^1= \bar{\phi}$ on $\partial S_0$ and 
\[ \Psi_0( \tilde{\varphi}^1  _{\rceil \partial T^1_0}) + \Psi_0(\tilde{\varphi}^1  _{\rceil \partial (S_0 \setminus T^1_0)}) \leq \Psi_0(\bar{\phi} _{\rceil \P_0}) + \frac{\eps}{4h}(t_1 - t_0).  \]

To conclude the step we need to repeat the very same argument on $\tilde{\varphi}^1  _{\rceil \partial (S_0 \setminus T^1_0)}$ by replacing $\P_0$ with $\U$ and considering a $\delta''_1$-modification of $\nu_{x_2}$,  where $\delta''_1$ is chosen so that 
\[  \delta''_1 < \min \left\lbrace \bar{\delta}(\U),  \frac{\eps(t_1 - t_0)}{32\ell\mathcal{H}^1(\U)}  \right\rbrace \ \text{and Lemma \ref{lemma: 2.5 Pratelli2} applies with $\delta''_1$ for $\U$ and }  \frac{\eps (t_1 - t_0)}{16h(\ell + h)}. \]
This would provide a continuous, injective and piecewise linear map $\tilde{\phi} :  \partial T_0^1 \cup \partial R_0 \cup  \partial T_0^2 \to \R^2$ extending $\tilde{\varphi}^1$ (hence, ultimately, $\bar{\phi}$) such that 
\begin{align*}
\Psi_0(\tilde{\phi} _{\rceil \partial T_0^1}) + \Psi_0(\tilde{\phi} _{\rceil \partial R_0}) + \Psi_0(\tilde{\phi} _{\rceil \partial T_0^2}) &\leq \Psi_0(\tilde{\varphi}^1  _{\rceil \partial (S_0 \setminus T^1_0)}) + \frac{\eps}{4h}(t_1 - t_0) \\
&\leq  \Psi_0(\bar{\phi} _{\rceil \P_0}) + \frac{\eps}{2h}(t_1 - t_0)
\end{align*}
thus proving \eqref{eq: estimate Psi function on T0l R0 T0r} and concluding the step.

\step{V}{Recursion and conclusion}{Skel5}

In this final step we want to conclude our construction by recursion.  In steps~\ref{Skel3} and \ref{Skel4} we divided $\Delta$ into a new convex polygon with two horizontal sides $\Delta^+$ and a horizontal strip $S_0$ given by a rectangle $R_0$ and two triangles $T^{1}_0, T^{2}_0$.


Then we defined continuous, injective, piecewise linear functions $\bar{\phi}, \tilde{\phi}$ such that
$\bar{\phi}=\tilde{\phi}$ on $\partial \Delta^+$ satisfies (by \eqref{eq: estimate Psi function on S0}, \eqref{eq: estimate Psi function on T0l R0 T0r} and our choice of $\delta_1$) 
\[ 
\begin{split}
\Psi_0(\bar{\phi} _{\rceil \partial S_0}) + \Psi_0(\bar{\phi} _{\rceil \partial \Delta^+}) &\leq \Psi_0(\varphi) + \frac{\eps}{2h} \big(t_1 - t_0 + 2\delta_1\big) \\
&\leq  \Psi_0(\varphi) +\frac{\eps}{2h} (t_1 - t_0) + \frac{\eps}{4}  \frac{1}{2},  \\
\Psi_0(\tilde{\phi} _{\rceil \partial T_0^1}) + \Psi_0(\tilde{\phi} _{\rceil \partial R_0}) + \Psi_0(\tilde{\phi} _{\rceil \partial T_0^2}) &\leq \Psi_0(\bar{\phi} _{\rceil \partial S_0}) + \frac{\eps}{2h} \big(t_1 - t_0 + 2\delta_1\big) \\
&\leq \Psi_0(\bar{\phi} _{\rceil \partial S_0}) +  \frac{\eps}{2h} (t_1 - t_0) + \frac{\eps}{4} \frac{1}{2}.
\end{split}
  \]
Iterating the construction and choosing for every $i$ the parameter $\delta_i$ suitably small depending on $t_{i-1}$, $\frac{h}{2^{i+1}}$ and the polygon $\P^+ = \P \setminus \bigcup_{j=0}^{i-1} \P_j$, we then find  
\[
\sum_{i=0}^{M-1} \Psi_0(\bar{\phi} _{\rceil \partial S_i})  \leq \Psi_0(\varphi) +  \frac{\eps}{2h} \sum_{i=0}^{M-1} (t_{i+1} - t_i) + \frac{\eps}{4} \sum_{i=0}^{M-1} \frac{1}{2^i} 
\]
and 
\[ \sum_{i=0}^{M-1 }\Psi_0(\tilde{\phi} _{\rceil \partial T_i^1}) + \Psi_0(\tilde{\phi} _{\rceil \partial R_i}) + \Psi_0(\tilde{\phi} _{\rceil \partial T_i^2}) \leq  \sum_{i=0}^{M-1 }\Psi_0(\bar{\phi} _{\rceil \partial S_i}) +  \frac{\eps}{2h} \sum_{i=0}^{M-1} (t_{i+1} - t_i) + \frac{\eps}{4} \sum_{i=0}^{M-1} \frac{1}{2^i},   \]
which finally imply \eqref{eq: estimate psi functions on strips} and \eqref{eq: estimate psi functions on skeleton} respectively. 

\end{proof}

\section{Piecewise affine extension}

In this subsection we investigate two possible finitely piecewise affine homeomorphic extension inside triangles.  

\begin{lemma}[Extension-direct]\label{lemma: extension up & bottom triangles}
Let $T \subset \R^2$ be a triangle of corners $A,B,C$ such that $BC$ is horizontal and $A^*$ is the intersection of $BC$ with the bisector of the angle at $A$. If $\phi : \partial T \to \R^2$ is continuous, injective and linear on each of the segments $AB, AC, BA^*$ and $A^*C$, then there exists a bi-affine homeomorphism $v: T \to \R^2$ such that $v = \phi$ on $\partial T$ and 
\begin{equation}\label{eq: estimate on top & bottom triangles}
\| Dv \|_0 (T) \leq \mathcal{H}^1(\phi(\partial T)) \mathcal{H}^1(\partial T).
\end{equation}
\end{lemma}
\begin{proof}
The proof is immediate,  indeed it is enough to consider the continuous map $v$ which is affine on each of the triangles $T_1 :=ABA^*$, $T_2 := BA^*C$.  Then we can compute 
\[  D_1 v _{\rceil T_1} = \frac{\phi(A^*) - \phi(B)}{(A^*)_1 - (B)_1}\L^2,\quad D_1 v _{\rceil T_2} = \frac{\phi(C) - \phi(A^*)}{(C)_1-(A^*)_1}\L^2,  \]
and
$$
	D_{2} v _{\rceil T_1} = \frac{\phi(A) - \phi(A^*) - D_1v_{\rceil T_1}((A)_1 -(A^*)_1)}{(A)_2-(A^*)_2}\L^2
$$
and
$$
	D_{2} v _{\rceil T_2}  =  \frac{\phi(A) - \phi(A^*) - D_1v_{\rceil T_2}((A)_1 -(A^*)_1)}{(A)_2-(A^*)_2}\L^2.
$$
In particular,  one can estimate
\begin{align*}
\| Dv \|_0 (T)  =& |D_1 v| (T_1) + |D_{2} v | (T_1)  + |D_1 v| (T_2) + |D_{2} v | (T_2) \\
 \leq&\frac{1}{2} |\phi(A^*) - \phi(B)|\big[|(A)_2-(A^*)_2| + |(A)_1 - (A^*)_1| \big]\\
 &+ \frac{1}{2} |\phi(A^*) - \phi(A)||(B)_1 - (A^*)_1| \\
 &+\frac{1}{2} |\phi(C) - \phi(A^*)|\big[|(A)_2-(A^*)_2| + |(A)_1 - (A^*)_1| \big]\\
 &+ \frac{1}{2} |\phi(A^*) - \phi(A)||(C)_1 - (A^*)_1| \\
\leq&  \mathcal{H}^1(\phi(\partial T))\frac{1}{2} \big( 4|A - A^*| + |B-C| \big) \\
\leq&  \mathcal{H}^1(\phi(\partial T))\frac{1}{2} \big( 2|A - B| + 2|A - C| + 2|B-C| \big) \\
\leq& \mathcal{H}^1(\phi(\partial T)) \mathcal{H}^1(\partial T).
\end{align*}
\end{proof}

\begin{figure}
\begin{tikzpicture}[line cap=round,line join=round,>=triangle 45,x=0.8cm,y=0.8cm]
\clip(0.5,0.5) rectangle (17.5,9.5);
\draw [shift={(1.,1.)},line width=0.5pt,color=qqwuqq,fill=qqwuqq,fill opacity=0.10000000149011612] (0,0) -- (0.:0.43137436300339144) arc (0.:45.:0.43137436300339144) -- cycle;
\fill[line width=0.5pt,color=qqqqcc,fill=qqqqcc,fill opacity=0.10000000149011612] (12.5,2.) -- (11.593534555543044,5.607074792135797) -- (10.024333735048703,6.532515496109497) -- (11.0886059695376,7.801865992760069) -- (14.278688614165699,8.281204835735275) -- (14.24,4.82) -- (17.16851183248274,3.5595506713107548) -- (14.56,2.46) -- (16.2,2.06) -- (12.,1.) -- cycle;
\fill[line width=0.5pt,color=qqwuqq,fill=qqwuqq,fill opacity=0.10000000149011612] (9.,9.) -- (9.,1.) -- (1.,1.) -- cycle;
\fill[line width=0.5pt,color=qqwuqq,fill=qqwuqq,fill opacity=0.3] (9.,7.) -- (9.,1.) -- (2.,1.) -- (8.,7.) -- cycle;
\fill[line width=0.5pt,color=qqwuqq,fill=qqwuqq,fill opacity=0.10000000149011612] (9.,8.) -- (9.,7.) -- (8.,7.) -- cycle;
\draw [line width=0.5pt,color=qqqqcc] (12.5,2.)-- (11.593534555543044,5.607074792135797);
\draw [line width=0.5pt,color=qqqqcc] (11.593534555543044,5.607074792135797)-- (10.024333735048703,6.532515496109497);
\draw [line width=0.5pt,color=qqqqcc] (10.024333735048703,6.532515496109497)-- (11.0886059695376,7.801865992760069);
\draw [line width=0.5pt,color=qqqqcc] (11.0886059695376,7.801865992760069)-- (14.278688614165699,8.281204835735275);
\draw [line width=0.5pt] (3,3)-- (9,3);
\draw [line width=0.5pt,color=qqqqcc] (14.278688614165699,8.281204835735275)-- (14.24,4.82);
\draw [line width=0.5pt,color=qqqqcc] (14.24,4.82)-- (17.16851183248274,3.5595506713107548);
\draw [line width=0.5pt,color=qqqqcc] (17.16851183248274,3.5595506713107548)-- (14.56,2.46);
\draw [line width=0.5pt,color=qqqqcc] (14.56,2.46)-- (16.2,2.06);
\draw [line width=0.5pt,color=qqqqcc] (16.2,2.06)-- (12.,1.);
\draw [line width=0.5pt,color=qqqqcc] (12.,1.)-- (12.5,2.);
\draw [line width=0.5pt,color=qqwuqq] (9.,9.)-- (9.,1.);
\draw [line width=0.5pt,color=qqwuqq] (9.,1.)-- (1.,1.);
\draw [line width=0.5pt,color=qqwuqq] (1.,1.)-- (9.,9.);
\draw [line width=0.5pt] (2.,1.)-- (9.,8.);
\draw [line width=0.5pt] (9.,7.)-- (8.,7.);
\draw [line width=0.5pt] (5.,5.)-- (5.,1.);
\draw [line width=0.5pt,dotted] (11.593534555543044,5.607074792135797)-- (11.978076027642164,6.459543886710959);
\draw [line width=0.5pt,dotted] (12.5,2.)-- (13.237707607409352,1.90524137817667);

\draw [line width=0.5pt,color=ffqqqq] (11.785805291592602,6.033309339423375)-- (11.122983404285261,5.884583689013925);

\draw [line width=0.5pt,color=ffqqqq] (12.2,1.4)-- (13.030024235644683,1.9319183297037836);
\draw [line width=0.5pt,color=ffqqqq] (11.785805291592604,6.033309339423379)-- (13.030024235644683,1.9319183297037836);

\draw [line width=0.5pt] (11.508524600980957,5.971092397214771)-- (10.799345098918286,6.075450311845103);

\draw [line width=0.5pt,color=qqwuqq] (9.,8.)-- (9.,7.);
\draw [line width=0.5pt,color=qqwuqq] (9.,7.)-- (8.,7.);
\draw [line width=0.5pt,color=qqwuqq] (8.,7.)-- (9.,8.);
\draw [->,line width=0.5pt] (10.4,5.6) -- (11.05,6.0);
\draw [->,line width=0.5pt] (11.3,3.7) -- (12.1,4.3);
\draw (13.,2.4) node[anchor=north west] {$\bX$};
\draw (11.7,6.4) node[anchor=north west] {$\bY$};
\draw (11.15,6.7) node[anchor=north west] {$\bZ$};
\begin{scriptsize}
\draw [fill]  (1.,1.) circle ( 1.5pt);
\draw  (0.7,0.8) node {$A$};
\draw [fill]  (9.,1.) circle ( 1.5pt);
\draw  (8.7,0.8) node {$B$};
\draw[fill]   (9.,9.) circle ( 1.5pt);
\draw  (9.098923911916296,9.270938609988487) node {$C$};
\draw[fill]   (12.5,2.) circle ( 1.5pt);
\draw  (11.9,2.) node {$\phi(A)$};
\draw[fill]   (11.593534555543044,5.607074792135797) circle ( 1.5pt);
\draw  (11.1,5.4) node {$\phi(C)$};
\draw[fill]   (2.,1.) circle ( 1.5pt);
\draw [fill] (1.7,0.8) node {$D$};
\draw [fill]  (9.,8.) circle ( 1.5pt);
\draw  (9.25,8.1) node {$E$};
\draw [fill]  (9.,7.) circle ( 1.5pt);
\draw  (9.25,7.1) node {$F$};
\draw  (8.7,7.35) node {$\tilde{T}$};
\draw [fill]  (8.,7.) circle ( 1.5pt);
\draw [fill]  (9.,3.) circle ( 1.5pt);
\draw [fill]  (4.,3.) circle ( 1.5pt);
\draw [fill]  (5.,5.) circle ( 1.5pt);
\draw [fill]  (5.,4.) circle ( 1.5pt);
\draw [fill]  (5.,1.) circle ( 1.5pt);
\draw [fill]  (3.,3.) circle ( 1.5pt);
\draw  (7.8,7.1) node {$G$};
\draw  (3.2,2.75) node {$H_t^1$};
\draw  (4.2,2.75) node {$H_t^2$};
\draw  (9.3,2.75) node {$H_t^3$};
\draw  (5.3,4.75) node {$V_s^3$};
\draw  (5.3,3.75) node {$V_s^2$};
\draw  (5.3,0.75) node {$V_s^1$};
\draw  (11.3,3.5) node {$\P_{ADEC}$};
\draw  (14.,3.5) node {$\P_{\Omega}$};
\draw  (10.3,5.5) node {$\P_{\tilde{T}}$};
\draw [fill]  (7.,6.) circle ( 1.5pt);
\draw  (6.7,6.) node {$Q$};
\draw  (7,3.5) node {$\Omega$};
\draw  (13.15,4.5) node {{\color{red}$\tilde{\gamma}_{\phi(D)\phi(E)}$}};
\draw [fill]  (3.,2.) circle ( 1.5pt);
\draw  (2.75,2.07) node {$P$};
\draw[fill]   (12.2,1.4) circle ( 1.5pt);
\draw  (11.7,1.3) node {$\phi(D)$};
\draw [fill]  (13.030024235644683,1.9319183297037836) circle ( 1.5pt);
\draw  [fill] (11.785805291592604,6.033309339423379) circle ( 1.5pt);
\draw [fill]  (11.508524600980957,5.971092397214771) circle ( 1.5pt);
\draw [fill] (10.799345098918282,6.075450311845105) circle ( 1.5pt);
\draw (10.3,5.99) node {$\phi(F)$};
\draw (1.6,1.2) node {$\beta$};
\end{scriptsize}
\end{tikzpicture}
  \caption{The decomposition of $T$ into $T = \Omega \cup \tilde{T} \cup ADEC$. And the corresponding decomposition of $\operatorname{int}\phi(\partial T)$ into $\P_{\Omega} \cup \P_{\tilde{T}} \cup \P_{ADEC}$}\label{Fig:Try}
\end{figure}

\begin{lemma}[Extension-indirect]\label{lemma: extension left & right triangles}
Let $T \subset \R^2$ be a triangle of corners $A,B,C$ such that $AB$ is horizontal,  $BC$ is vertical and let $\phi : \partial T \to \R^2$ be a continuous, piecewise linear, injective map such that $\phi$ is linear on the hypotenuse $AC$. 
For every $\eps >0$ there exists a finitely piecewise affine homeomorphism $v: T \to \R^2$ such that   $v = \phi$ on $\partial T$ and 
\begin{equation}\label{eq: estimate on left & right triangles}
\| Dv \|_0 (T) \leq \Psi_0(\phi)  + 242 \mathcal{H}^1(\partial T) \mathcal{H}^1(\phi(AC)) + \eps
\end{equation}
\end{lemma}
\begin{proof}

Let $\eps > 0$ be fixed arbitrary small.   

For simplicity of notation, through the proof we refer to $\beta$ as the internal angle of the corner $A = (0,0)$ and we will denote $\bd := |\varphi(A) - \varphi(C)| = \mathcal{H}^1(\varphi(AC))$ and $d = |A - C|$.   Clearly, since the internal angle in $B$ is $\pi/2$,  then $\beta \in (0,\pi/2)$. 

Since the polygon of boundary $\varphi(\partial T)$ is non degenerate,  then Definition \ref{def: geodesics and modified geodesics} provides some constant $\bar{\delta} > 0$,  then we consider 
\begin{equation}\label{eq: definition of eta}
\eta < \Big\{  1,  \frac{d}{2}, \frac{\bar{\delta}}{4}, \frac{\bd}{14}, \frac{\eps}{4}, \big( 1 + \frac{1}{2\tan\beta} \big)^{-1}, \|D\phi\|_{\infty}^{-1}\Big\}.
\end{equation}

The basic idea of the proof is to find a suitable one-dimensional skeleton $\Upsilon$ inside $T$,  construct a continuous, piecewise linear and injective map $\tilde{\varphi} : \Upsilon \to \R^2$ coinciding with $\varphi$ on $\partial T$ and finally perform a suitable piecewise affine extension inside each component of the partition of $T$ identified by $\Upsilon$. 

For clarity, we present the proof in three separate steps. 

\step{I}{Definition of a first skeleton $\Xi$ and a continuous piecewise linear injective map $\varphi^1 : \Xi \to \R^2$}{EIS1}

In this step we would like to construct a one-dimensional skeleton of the form $\Xi = \partial T \cup DE \cup FG$,  for some suitably chosen points $D,E,F,G$, and we will define an extension $\varphi^1$ of $\varphi$ on $DE \cup FG$ that is still continuous, piecewise linear and injective. See Figure~\ref{Fig:Try} for an illustration.

We will select $D,E,F,G$ so that 
\[  D \in AB, \quad E,F \in BC, \quad G \in DE, \quad  FG \parallel AB \quad \text{  and } \quad DE \parallel AC,  \]
and satisfying the following estimates 
\begin{equation}\label{eq: estimates from Step I}
\begin{split}
&|A - D|< \eta \quad \text{ and } \quad  |C- E| < \eta,  \\
&\mathcal{H}^1(\varphi^1(\partial \tilde{T})) < 4\eta, \\
&\Psi_0(\varphi^1 _{\rceil \partial \Omega})\leq \Psi_0(\varphi) +(\bd + 14\eta) \mathcal{H}^1(\partial T). 
\end{split}
\end{equation}
where $\tilde{T} \subset T$ is the triangle of corners $E,F,G$ and $\Omega \subset T$ is the trapezoid of corners $D, B, F, G$. 

Having fixed $\eta$, by the assumptions on $\varphi$, we can choose $D \in AB$ and $E \in BC$ so that 
\begin{itemize}
\item[i)] $|A - D| < \eta$ and $|C- E| < \eta$;
\item[ii)] $|\varphi(A) - \varphi(D)| < \eta$ and $|\varphi(C) - \varphi(E)| < \eta$;
\item[iii)] the restriction of $\varphi$ is linear on $AD$ and $CE$;
\item[iv)] $DE$ is parallel to $AC$;
\item[v)] the point $\bX$ is on the internal bisector of $\varphi(A)$ and $\bY$ on the internal bisector of $\varphi(C)$ such that $|\varphi(A) - \bX| < 2\eta$ and  $|\varphi(C) - \bY| < 2\eta$ and the piecewise linear path $\tilde{\gamma}_{\varphi(D) \varphi(E)} := \varphi(D) \bX \bY \varphi(E)$ lies in the interior of $\varphi(\partial T)$ and is a $\bar{\delta}/2$-modification of the geodesic $\gamma_{\varphi(D) \varphi(E)}$ in the sense of Definition \ref{def: geodesics and modified geodesics}.
\end{itemize}

Observe, that ii) and v) imply that $|\varphi(D) - \bX|,  |\varphi(E) - \bY| < 3\eta$ and also 
\begin{equation}\label{eq: estimate path from varphiD to varphiE}
\begin{split}
\mathcal{H}^1(\tilde{\gamma}_{\varphi(D) \varphi(E)}) &< |\varphi(D) - \bX| + |\bX - \bY| + |\varphi(E) - \bY| \\
&< 6\eta + |\varphi(A) - \bX|  + |\varphi(A) - \varphi(C)| + |\varphi(C) - \bY| \\
&< \bd + 10\eta.
\end{split}
\end{equation}

We find a point $F$ on the segment $EB$, a point $G\in DE$ and a point $\bZ \in [\bY \varphi(E)]$ such that
\begin{itemize}
\item[vi)] $|F-E | < \eta$ and $|\varphi(F) - \varphi(E)| < \eta$;
\item[vii)] $\varphi$ is linear on $EF$;
\item[viii)] $|\varphi(E) - \bZ|< \eta$ and $[\varphi(F)\bZ]$ lies in $\operatorname{int}\varphi(\partial T)$. As a consequence $|\varphi(F) - \bZ| < 2\eta$
\item[ix)] $G_2 = F_2$, $|G - E| < \eta (\sin \beta)^{-1}$ and $|G-F| < \eta (\tan\beta)^{-1}$.
\end{itemize} 

This concludes the definition of $\Xi$. Indeed i) ensures the first equation of \eqref{eq: estimates from Step I}, then $\tilde{T}$ is a right-angle triangle and $\Omega$ is a trapezoid inside $T$. 

We now proceed to construct a function $\varphi^1 : \Xi \to \R^2$ extending $\varphi$ such that the second and third estimates of  \eqref{eq: estimates from Step I} are satisfied.  In order to do that, we consider two further auxiliary points $P,Q \in DG$ so that 
\begin{itemize}
\item[x)] $|P-D| < \eta$ and $|Q-G| < \eta$;
\end{itemize}
and we set 
\[ \varphi^1(P) := \bX, \quad \varphi^1(Q) := \bY \quad \text{ and } \quad \varphi^1(G) := \bZ.  \]
We then define $\varphi^1 : \Xi \to \R^2$ so that $\varphi^1 = \varphi$ on $\partial T$,  $\varphi^1_{\rceil DP}$ is the parametrization at constant speed of the segment $\varphi(D)\bX$,  $\varphi^1_{\rceil PQ}$ is the parametrization at constant speed of the segment $\bX\bY$,  $\varphi^1_{\rceil QG}$ is the parametrization at constant speed of the segment $\bY\bZ$, $\varphi^1_{\rceil GE}$ is the parametrization at constant speed of the segment $\bZ\varphi(E)$ and, finally,  $\varphi^1_{\rceil GF}$ is the parametrization at constant speed of the segment $\bZ\varphi(F)$. A sketch of the situation is presented in Figure \ref{Fig:Try}

The second estimate of  \eqref{eq: estimates from Step I} is a direct consequence of vi) and viii),  indeed $\varphi^1$ is linear in each of the segments $EF, FG, EG$ and by triangular inequality we have 
\begin{align*}
\mathcal{H}^1(\varphi^1(\partial \tilde{T})) &= |\varphi(E) - \varphi(F)| +  |\varphi(E) - \bZ| +  |\bZ - \varphi(F)| \\
& \leq 2\big( |\varphi(E) - \varphi(F)| +  |\varphi(E) - \bZ|\big) \leq 4\eta. 
\end{align*}

The remaining part of the step is devoted to the proof of the third estimate of \eqref{eq: estimates from Step I}. In the following $X_i$ is the $i$-th coordinate of the point $X$. For every $t \in [0, E_1]$ we denote by $H_t^1, H_t^2, H_t^3$ the intersections between the horizontal line $\R \times \{t \}$ and the curves $AC, DE, BF$ respectively.  Clearly,  $H_0^1 = A, H^2_0 = D,  H_0^3 = B,  H_{F_2}^2 = G$ and $H_{F_2}^3 = F$.
Similarly, for every $s \in [D_1,B_1]$ we denote by $V_s^1, V_s^2, V_s^3$ the intersections between the vertical line $\{s\} \times \R$ and the sets $DB,  DG\cup GF, AC$ respectively. Then $V_{D_1}^1 = V_{D_1}^2 = D, V_{B_1}^3 = C, V_{B_1}^2 = F$ and  $V_{B_1}^1 = B$. 

We recall that, by definition,  $\Psi_0(\varphi^1 _{\rceil \partial \Omega})$ corresponds to the following quantity 
\[  \int_0^{F_1} \rho_{\P_{\Omega}} (\varphi^1(H_t^2), \varphi^1(H_t^3)) dt + \int_{D_1}^{B_1}  \rho_{\P_{\Omega}} (\varphi^1(V_s^1), \varphi^1(V_s^2)) ds,   \]
where $\P_{\Omega}$ is the non degenerate polygon identified by $\varphi^1(\partial \Omega)$. 

Notice that, by construction,  the curve $\varphi^1(DG)$ is exactly the piecewise linear curve $\varphi(D)\bX\bY\bZ$, so $\varphi^1(H_t^2)$ and $\varphi^1(V_s^2)$ will lie on  $\varphi(D)\bX\bY\bZ$ for every $t \in [0,F_2]$ and $s \in [D_1, G_1]$. On the other hand, when $s \in [G_1, B_1]$ we get that $\varphi^1(V_s^2)$ lies on the segment $\bZ\varphi(F)$. 

Obviously one can construct a path in $\P_{\Omega}$ from $\varphi^1(H_t^3)$ to $\varphi^1(H_t^2)$ by following $\gamma_{\varphi(H_t^1)\varphi(H_t^3)} \cap \P_{\Omega}$ and when necessary going around $\P_{\tilde{T}}$ on its boundary and travelling along $\tilde{\gamma}_{\varphi(D)\varphi(E)}$ till one gets to $\varphi^1(H_t^2)$. The length of this curve bounds the length of the geodesic in $\P_{\Omega}$ between $\varphi^1(H_t^2)$ and $\varphi^1(H_t^3)$. Further 
\begin{equation}\label{Johnny}
\begin{aligned}
 \rho_{\P_{\Omega}} \big( \varphi^1(H_t^2),   \varphi^1(H_t^3) \big) &\leq \mathcal{H}^1(\gamma_{\varphi(H_t^1)\varphi(H_t^3)} \cap \P_{\Omega}) + \mathcal{H}^1(\partial \P_{\tilde{T}}) +  \mathcal{H}^1 (\tilde{\gamma}_{\varphi(D)\varphi(E)}) \\
 & \leq \rho_{\varphi(\partial T)} \big(\varphi(H_t^1),\varphi(H_t^3) \big) + \bd + 14\eta, 
\end{aligned}
\end{equation}
where in the last inequality we used the second of \eqref{eq: estimates from Step I} and \eqref{eq: estimate path from varphiD to varphiE}.
Analogously, for every $s \in (D_1,B_1)$ it holds that 
\begin{equation}\label{Jimmy}
\begin{aligned}
 \rho_{\P_{\Omega}} \big( \varphi^1(V_s^1),   \varphi^1(V_s^2) \big)& \leq \mathcal{H}^1(\gamma_{\varphi(V_s^1)\varphi(V_s^3)} \cap \P_{\Omega}) + \mathcal{H}^1(\partial \P_{\tilde{T}}) + \mathcal{H}^1 (\tilde{\gamma}_{\varphi(D)\varphi(E)}) \\
 & \leq \rho_{\varphi(\partial T)} \big(\varphi(V_s^1),\varphi(V_s^3) \big) + \bd + 14\eta.
\end{aligned}
\end{equation}

Gathering the last two estimates together we obtain 
\begin{equation}\label{Jack}
\begin{aligned}
\Psi_0 (\varphi^1 _{\rceil \partial \Omega}) &= \int_{0}^{F_2} \rho_{\P_{\Omega}} \big( \varphi^1(H_t^2),   \varphi^1(H_t^3) \big) dt + \int_{D_1}^{B_1} \rho_{\P_{\Omega}} \big( \varphi^1(V_s^1),   \varphi^1(V_s^2) \big) ds \\
&\leq  \int_0^{F_2} \rho_{\varphi(\partial T)} \big(\varphi(H_t^1),\varphi(H_t^3) \big) dt + \int_{D_1}^{B_1}  \rho_{\varphi(\partial T)} \big(\varphi(V_s^1),\varphi(V_s^3) \big) ds \\
&\quad + (\bd + 14\eta) (|F-B| + |D-B|)\\
&\leq \Psi_0(\varphi) + (\bd + 14\eta) \mathcal{H}^1(\partial T). 
\end{aligned}
\end{equation}
which is exactly the third estimate of \eqref{eq: estimates from Step I}.

\step{II}{Definition of the final skeleton $\Upsilon$ and the continuous piecewise linear injective map $\tilde \varphi : \Upsilon \to \R^2$}{EIS2}

The aim of this step is to define a set $\Upsilon$ depicted in Figure~\ref{Fig:Conquer} and a map $\tilde{\varphi}$ satisfying the following properties: 
\begin{itemize}
\item[1)] the set $\Upsilon$ subdivides $T$ in the essentially disjoint union $\hat{T}\cup\tilde{T} \cup \mathbb{P} \cup  \bigcup_{i=0}^{M-1} R_i$, where $\mathbb{P}$ is a polygon which is $\tilde{C}$-bi-Lipschitz ($\tilde{C}$ is a universal constant) equivalent to the rectangle $[0,|A-C|]\times [0,|A-D|\sin\beta]$ , $\hat{T}$ is a triangle near either $A$ or $C$, $\tilde{T}$ is the triangle defined in step~\ref{EIS1} of corners $E,F,G$ and $R_i$ are $M$ pairwise essentially disjoint rectangles having horizontal and vertical sides. 

\item[2)] $\tilde{\varphi} : \Upsilon \to \R^2$ is piecewise linear, continuous and injective, moreover it coincides with $\varphi^1$ on $\Xi$ (where $\Xi$ is the set defined in step~\ref{EIS1}) and satisfies the following estimates
\begin{equation}\label{eq: estimate Psi function on rectangles}
\sum_{i=0}^{M-1} \Psi_0 \big( \tilde{\varphi} _{\rceil \partial R_i} \big) \leq \Psi_0 (\varphi) + (\bd + 14\eta) \mathcal{H}^1(\partial T) + \frac{\eta}{50} .
\end{equation}
\end{itemize}


We use Lemma \ref{lemma: skeleton strips} on $\phi^1$ and on $\Omega$ with
\[   \xi < \frac{1}{100}  \min  \left\lbrace \sin\beta\dist (D, AC),  \frac{\eta}{2d} \right\rbrace \]
and get a number $M \in \N$ and $M$ values 
\[  0= t_0 < t_1 < \ldots < t_{M-1} < t_M = F_2 \]
such that $|t_{i+1} - t_i| < \xi$ for every $i = 1 \dots M-1$. Moreover it is possible to decompose $\Omega$ into the essentially disjoint union of $M-1$ horizontal strips $S_i = \R \times[t_{i-1}, t_i] \cap \Omega$ such that $\varphi^1$ is linear on $I_i := \partial S_i \cap DG$ and on $ \partial S_i \cap FB$  and $\mathcal{H}^1(\varphi^1(I_i)) \leq \xi$. Further we deduce from Lemma~\ref{lemma: skeleton strips} the existence of a continuous, piecewise linear and injective 
\[ \varphi^2 : \partial T \cup DE \cup FG \cup \bigcup_{i=1}^{M-1} \partial S_i \to \R^2   \]
such that $\varphi^2 = \varphi^1$ on $ \partial T \cup DE \cup FG$ and 
\begin{equation*}
\sum_{i=0}^{M-1} \Psi_0 \big( \varphi^2 _{\rceil \partial S_i} \big) \leq \Psi_0 \big( \varphi^1 _{\rceil \partial \Omega} \big) + \xi. 
\end{equation*}
A consequence of the above inequality, the choice of $\xi$ and the third estimate of \eqref{eq: estimates from Step I} is that 
\begin{equation}\label{eq: estimate Psi function on strips}
\begin{split}
\sum_{i=0}^{M-1} \Psi_0 \big( \varphi^2 _{\rceil \partial S_i} \big) &\leq \Psi_0 (\varphi) + (\bd + 14\eta) \mathcal{H}^1(\partial T) + \xi \\
&< \Psi_0 (\varphi) + (\bd + 14\eta) \mathcal{H}^1(\partial T) + \frac{\eta}{100}.
\end{split}
\end{equation}

We now proceed to construct the skeleton $\Upsilon$ and the final map $\tilde{\varphi}$. Observe that for every $i = 0 \ldots M-1$,  the right-angle triangle $T_i$ of hypotenuse $I_i$ constructed at the exterior of $\Omega$ is still contained in $T$ and does not intersect $AC$.  Indeed, by construction, one has that $\mathcal{H}^1(\partial S_i \cap FB) = |t_{i+1} - t_i| < \xi$ and hence $\mathcal{H}^1(I_i) < \frac{\xi}{\sin \beta}$ and the distance between any point of $T_i$ and the segment $AC$ must be at least $\dist (D, AC) - \frac{\xi}{\sin\beta} > \frac{99}{100} \dist (D, AC)$,  thus ensuring that $T_i \subset T$.  Clearly,  $T_i \cap T_j$ is either empty or contains at most one corner (this corresponds to the case where $j = i \pm 1$). 

 \begin{figure}
  \begin{tikzpicture}[line cap=round,line join=round,>=triangle 45,x=1.0cm,y=1.0cm]
  \clip(0.5,0.5) rectangle (9.5,9.5);
  \fill[line width=0.5pt,color=qqqqcc,fill=qqqqcc,fill opacity=0.10000000149011612] (12.5,2.) -- (11.593534555543044,5.607074792135797) -- (10.024333735048703,6.532515496109497) -- (11.0886059695376,7.801865992760069) -- (14.278688614165699,8.281204835735275) -- (14.24,4.82) -- (17.16851183248274,3.5595506713107548) -- (14.56,2.46) -- (16.2,2.06) -- (12.,1.) -- cycle;
  \fill[line width=0.5pt,color=qqwuqq,fill=qqwuqq,fill opacity=0.10000000149011612] (9.,9.) -- (9.,1.) -- (1.,1.) -- cycle;
  \fill[line width=0.5pt,color=qqwuqq,fill=qqwuqq,fill opacity=0.10000000149011612] (9.,8.) -- (9.,7.) -- (8.,7.) -- cycle;
  \draw [line width=0.5pt,color=qqqqcc] (12.5,2.)-- (11.593534555543044,5.607074792135797);
  \draw [line width=0.5pt,color=qqqqcc] (11.593534555543044,5.607074792135797)-- (10.024333735048703,6.532515496109497);
  \draw [line width=0.5pt,color=qqqqcc] (10.024333735048703,6.532515496109497)-- (11.0886059695376,7.801865992760069);
  \draw [line width=0.5pt,color=qqqqcc] (11.0886059695376,7.801865992760069)-- (14.278688614165699,8.281204835735275);
  \draw [line width=0.5pt,color=qqqqcc] (14.278688614165699,8.281204835735275)-- (14.24,4.82);
  \draw [line width=0.5pt,color=qqqqcc] (14.24,4.82)-- (17.16851183248274,3.5595506713107548);
  \draw [line width=0.5pt,color=qqqqcc] (17.16851183248274,3.5595506713107548)-- (14.56,2.46);
  \draw [line width=0.5pt,color=qqqqcc] (14.56,2.46)-- (16.2,2.06);
  \draw [line width=0.5pt,color=qqqqcc] (16.2,2.06)-- (12.,1.);
  \draw [line width=0.5pt,color=qqqqcc] (12.,1.)-- (12.5,2.);
  \draw [line width=0.5pt,color=qqwuqq] (9.,9.)-- (9.,1.);
  \draw [line width=0.5pt,color=qqwuqq] (9.,1.)-- (1.,1.);
  \draw [line width=0.5pt,color=qqwuqq] (1.,1.)-- (9.,9.);
  \draw [line width=0.5pt] (2.,1.)-- (9.,8.);
  \draw [line width=0.5pt] (9.,7.)-- (8.,7.);
  \draw [line width=0.5pt,dotted] (11.593534555543044,5.607074792135797)-- (11.978076027642164,6.459543886710959);
  \draw [line width=0.5pt,dotted] (12.5,2.)-- (13.237707607409352,1.90524137817667);
  \draw [line width=0.5pt,color=ffqqqq] (12.2,1.4)-- (13.030024235644683,1.9319183297037836);
  \draw [line width=0.5pt,color=ffqqqq] (11.785805291592599,6.033309339423368)-- (13.030024235644683,1.9319183297037836);
  \draw [line width=0.5pt,color=ffqqqq] (11.785805291592599,6.033309339423368)-- (11.122983404285266,5.884583689013922);
  \draw [line width=0.5pt] (11.508524600980945,5.971092397214763)-- (10.799345098918291,6.0754503118451);
  \draw [line width=0.5pt,color=qqwuqq] (9.,8.)-- (9.,7.);
  \draw [line width=0.5pt,color=qqwuqq] (9.,7.)-- (8.,7.);
  \draw [line width=0.5pt,color=qqwuqq] (8.,7.)-- (9.,8.);
  \draw (12.829803016637632,2.496576185377244) node[anchor=north west] {$\bX$};
  \draw (12.015789576762625,6.245057828258506) node[anchor=north west] {$\bY$};
  \draw (11.342469817853669,6.506346092909747) node[anchor=north west] {$\bZ$};
  \draw [line width=0.5pt] (2.5,1.5)-- (2.5,2.);
  \draw [line width=0.5pt] (2.5,2.)-- (9.,2.);
  \draw [line width=0.5pt] (3.,2.)-- (3.,2.5);
  \draw [line width=0.5pt] (3.,2.5)-- (9.,2.5);
  \draw [line width=0.5pt] (3.5,2.5)-- (3.5,3.);
  \draw [line width=0.5pt] (3.5,3.)-- (9.,3.);
  \draw [line width=0.5pt] (4.,3.)-- (4.,3.5);
  \draw [line width=0.5pt] (4.,3.5)-- (9.,3.5);
  \draw [line width=0.5pt] (4.5,3.5)-- (4.5,4.);
  \draw [line width=0.5pt] (4.5,4.)-- (9.,4.);
  \draw [line width=0.5pt] (5.,4.)-- (5.,4.5);
  \draw [line width=0.5pt] (5.,4.5)-- (9.,4.5);
  \draw [line width=0.5pt] (5.5,4.5)-- (5.5,5.);
  \draw [line width=0.5pt] (5.5,5.)-- (9.,5.);
  \draw [line width=0.5pt] (6.,5.)-- (6.,5.5);
  \draw [line width=0.5pt] (6.,5.5)-- (9.,5.5);
  \draw [line width=0.5pt] (6.5,5.5)-- (6.5,6.);
  \draw [line width=0.5pt] (6.5,6.)-- (9.,6.);
  \draw [line width=0.5pt] (2.,1.)-- (2.,1.5);
  \draw [line width=0.5pt] (2.,1.5)-- (9.,1.5);
  \draw [line width=0.5pt] (7.,6.)-- (6.9968483317784385,6.4968483317784385);
  \draw [line width=0.5pt] (6.9968483317784385,6.4968483317784385)-- (9.,6.4968483317784385);
  \draw [line width=0.5pt] (7.4968483317784385,6.4968483317784385)-- (7.49901242680979,6.99901242680979);
  \draw [line width=0.5pt] (7.49901242680979,6.99901242680979)-- (9.,7.);
  \draw [line width=0.5pt] (1.5,1.)-- (1.25,1.25);
  
  \fill[line width=0.5pt,color=qqwuqq,fill=qqwuqq,fill opacity=0.2] (9.,5.5) -- (9.,5.) -- (6.,5.) -- (6.5,5.5) --cycle;
  \fill[line width=0.5pt,color=qqwuqq,fill=qqwuqq,fill opacity=0.2] (9.,4.5) -- (9.,4.) -- (5.,4.) -- (5,4.5) --cycle;
  \begin{scriptsize}
  \draw [fill=black] (1.,1.) circle (1.5pt);
  \draw [fill=black] (9.,1.) circle (1.5pt);
  \draw [fill=black] (9.,9.) circle (1.5pt);
  \draw [fill=black] (9.,8.) circle (1.5pt);
  \draw [fill=black] (9.,7.) circle (1.5pt);
  \draw [fill=black] (8.,7.) circle (1.5pt);
  \draw [fill=black] (3.,2.) circle (1.5pt);
  \draw [fill=black] (7.,6.) circle (1.5pt);
  \draw [fill=black] (2.,1.) circle (1.5pt);

  \draw  (0.7,0.8) node {$A$};
 
  \draw  (8.7,0.8) node {$B$};
 
  \draw  (9.098923911916296,9.270938609988487) node {$C$};
   \draw  (11.9,2.) node {$\phi(A)$};
   \draw  (11.1,5.4) node {$\phi(C)$};
   \draw [fill] (2,0.8) node {$D$};
   \draw  (9.25,8.1) node {$E$};
   \draw  (9.25,7.1) node {$F$};
  \draw  (8.7,7.35) node {$\tilde{T}$};
  \draw  (1.0,1.3) node {$\hat{T}$};
   \draw  (7.8,7.3) node {$G$};
  \draw  (6.7,6.2) node {$Q$};
  \draw  (2.75,2.2) node {$P$};
  \draw  (7.7,5.2) node {$S_i$};
  \draw  (7.7,4.2) node {$R_{i-2}$};
  \draw  (6.15,5.3) node {$T_i$};
  \draw  (4.7,4.3) node {$\mathbb{P}$};
  \end{scriptsize}
  \end{tikzpicture}
  \caption{The division of $T$ into $T = \mathbb{P}\cup\tilde{T}\cup\bigcup_iR_i$ and that $R_i = T_i\cup S_i$ where one vertex of $R_i$ lies on $DE$. The set $\mathbb{P}$ is uniformly bi-Lipschitz equivalent to the rectangle $[0,|A-C|]\times [0,|A-D|]$. The usage of $\hat{T}$ is optional and is used near either $A$ or $C$ when the hypotenuse is close to being either horizontal or vertical.}
  \label{Fig:Conquer}
\end{figure}

Then, connecting the horizontal and vertical sides of $T_i$ for every $i = 0 \ldots M-1$, we obtain a continuous piecewise linear path $\Gamma$ connecting $D$ and $G$ that lies inside $T$ and, at the same time,  outside the trapezoid $\Omega$.  Moreover, by construction, we have that
\[  \dist (AC, \Gamma) \geq \frac{99}{100} \dist (D, AC) = \frac{99}{100} |A-D|\sin\beta  \]
 and 
\begin{equation}\label{eq: length of Gamma}
\mathcal{H}^1(\Gamma) \leq 2|D-G| + |E-G| \leq 2|D-E| \leq 2h.
\end{equation} 

Assuming that $\beta$ is bounded away from 0 and $\tfrac{\pi}{2}$ we have that the non degenerate polygon $\mathbb{P} \subset T$ of boundary $AC \cup AD \cup \Gamma \cup GE \cup EC$ is $\tilde{C}$-bi-Lipschitz equivalent to a rectangle of side-lengths $|A-C|=d$ and $|A-D|\sin\beta$. On the other hand if $\beta$ is very close to 0 it suffices to take away the triangle $\hat{T}$ with vertexes at $A$, $A/2 + D/2$ and third vertex on $AC$ with angle $\pi/2$. Then the remaining part of $\mathbb{P}$ is again $\tilde{C}$-bi-Lipschitz equivalent to a rectangle of side-lengths $|A-C|=d$ and $|A-D|\sin\beta$. Similarly if $\beta$ is close to $\pi/2$ we subtract a triangle $\hat{T}$ close to $C$ and then the remaining set is $\tilde{C}$-bi-Lipschitz equivalent to a rectangle of side-lengths $|A-C|=d$ and $|A-D|\sin\beta$.

Let us now observe that, by construction, the polygon of boundary $DB \cup FB \cup FG \cup \Gamma$ can be seen as the essentially disjoint union of  $M$ rectangles $R_i$ of horizontal and vertical sides such that $S_i = R_i \cap \Omega$ and $T_i = R_i \setminus \Omega$.  

Notice that the rectangles $R_i$ are pairwise essentially disjoint, moreover they are essentially disjoint from $\mathbb{P}$ and $\tilde{T}$.  We are finally in position to define the set
\[ \Upsilon := \partial \mathbb{P} \cup \partial \hat{T}\cup \partial \tilde{T} \cup \bigcup_{i=0}^{M-1} \partial R_i,  \]
 satisfying all conditions of property 1). 
 
We now focus on the definition of $\tilde{\varphi} : \Upsilon \to \R^2$. 
We set 
\[\tilde{\varphi} := \varphi^2 \quad \text{ on } \quad  \left( \Xi \cup \bigcup_{i=0}^{M-1} \partial S_i \right) \setminus DG, \]
 then we only need to care about the definition of $\tilde{\varphi}$ on $\Gamma$.  The idea is to let $\tilde{\varphi}(\partial T_i \setminus I_i) = \varphi^2 (I_i)$ and decide the parametrization in such a way that $\Psi_0(\tilde{\varphi} _{\rceil \partial R_i}) \lessapprox \Psi_0(\varphi^2 _{\rceil \partial S_i})$. 

For this reason,  for every $i = 0 \ldots M-1$,  we define $\tilde{\varphi} : \partial T_i \setminus I_i \to \R^2$ as the bi-linear map that parametrizes the segment $\varphi^2(I_i)$ at constant speed. 
We recall that Lemma \ref{lemma: skeleton strips} gives a $\varphi^2$ that is linear and equal to $\varphi^1$ on each $I_i$, and $\mathcal{H}^1(\varphi^2(I_i)) < \xi$.  
As a consequence,  for any choice of $X \in \Gamma \cap \partial T_i$ and $Y \in I_i$ we get that $| \tilde{\varphi}(X) -  \varphi^2(Y)| < \xi$,  and since any horizontal/vertical slice of $R_i$ intersects $\partial S_i$, it is immediate to deduce that 
\begin{equation*}
\Psi_0 \big( \tilde{\varphi} _{\rceil \partial R_i} \big) \leq  \Psi_0 \big(\varphi^2 _{\rceil \partial S_i} \big) + \xi \mathcal{H}^1(\Gamma \cap  \partial T_i) 
\end{equation*}
for every $i = 0 \ldots M-1$. 

Recalling the first of \eqref{eq: estimates from Step I},  \eqref{eq: length of Gamma} and the bound on $\xi$,  thanks to \eqref{eq: estimate Psi function on strips} we deduce 
\[
\begin{split}
\sum_{i=0}^{M-1} \Psi_0 \big( \tilde{\varphi} _{\rceil \partial R_i} \big) &\leq \sum_{i=0}^{M-1} \Big( \Psi_0 \big(\varphi^2 _{\rceil \partial S_i} \big) + \xi \mathcal{H}^1(\Gamma \cap  \partial T_i)  \Big) \\
&\leq  \Psi_0 (\varphi) + (\bd + 14\eta) \mathcal{H}^1(\partial T) + \frac{\eta}{100} + \xi \mathcal{H}^1(\Gamma) \\
&\leq  \Psi_0 (\varphi) + (\bd + 14\eta) \mathcal{H}^1(\partial T) + \frac{\eta}{50}
\end{split}
\]
which is exactly \eqref{eq: estimate Psi function on rectangles}.
Then the fact that $\tilde{\varphi}$ is continuous, injective and coincides with $ \varphi^2 = \varphi^1$ on $\Xi \setminus DG$,  implies property 2) and hence the conclusion of the step.

\step{3}{Piecewise affine extensions in the components of $T \setminus \Upsilon$ and conclusion}{EIS3}

In this conclusive step we perform independently piecewise affine extensions in the different components of $T \setminus \Upsilon$.  Indeed,  thanks to step~\ref{EIS2} we have that $T$ is the disjoint union of $\mathbb{P}$, $\tilde{T}$, possibly $\hat{T}$ and $M$ rectangles $R_i$ and $\tilde{\varphi}$ is continuous, injective and piecewise linear on the respective boundaries. This allows to work separately in each component with piecewise affine extensions that coincide with the restriction of $\tilde{\varphi}$ on the boundary of the considered component.  

Let us first consider the rectangle $R_i$ for some $i = 0 \ldots M-1$.  
Proposition \ref{prop: straight minimal extension} applied to $R_i$ and $\tilde{\varphi}_{\rceil \partial R_i}$ with parameter $\frac{\eps}{2^i}$ provides a finitely piecewise affine homeomorphism $v_i : R_i \to \R^2$ extending $\tilde{\varphi}$ on $\partial R_i$ such that 
\begin{equation}\label{eq: estimate of the extension inside rectangle}
\| Dv_i \|_0 (R_i)\leq \Psi_0 (\tilde{\varphi}_{\rceil \partial R_i}) + \frac{\eps}{2^i}.
\end{equation}

We pass now to consider the extension inside the triangle $\tilde{T}$.  Being $\tilde{\varphi}$ linear on each side of $\partial \tilde{T}$, then we define $w: \tilde{T} \to \R^2$ as the unique affine extension of the boundary value $\tilde{\varphi}$. 
We can then directly compute 
\[  D_1 w = \frac{\tilde{\varphi}(F) - \tilde{\varphi}(G)}{|F-G|} \quad \text{ and } \quad D_2 w  = \frac{\tilde{\varphi}(E) - \tilde{\varphi}(F)}{|E-F|}.  \]
Therefore,  recalling vi), vii), viii), ix) and the fact that $\tilde{\varphi}(G) = \bZ, \tilde{\varphi}(E) = \varphi(E), \tilde{\varphi}(F) = \varphi(F)$ from step~\ref{EIS1},  we get 
\begin{equation}\label{eq: estimate of the extension inside small triangle}
\begin{split}
\| Dw \|_0(\tilde{T}) &= \frac{1}{2} \left(|E-F| |\tilde{\varphi}(G) - \tilde{\varphi}(F)| + |F-G| |\tilde{\varphi}(E) - \tilde{\varphi}(F)| \right) \\
&\leq \frac{\eta}{2} \big(  |\bZ - \tilde{\varphi}(F)| + |F-G| \big) \leq \eta^2 \left( 1 + \frac{1}{2\tan\beta} \right).
\end{split}
\end{equation}

To extend in $\hat{T}$  we use Lemma~\ref{lemma: extension up & bottom triangles}. We have $\H^1(\partial \hat{T})\leq \tilde{C} \eta$ and $\H^1(\tilde{\phi}(\partial\hat{T})) \leq \tilde{C} \|D\phi\|_{\infty}\eta$. Recalling \eqref{eq: definition of eta} we get a bi-affine homeomorphism $v$ equal to $\tilde{\phi}$ on $\partial\hat{T}$ and $\| Dv \|_0 (T) \leq \eta \leq \epsilon.$

At last, we discuss the extension inside $\mathbb{P}$.  Since $\mathbb{P}$ is $\tilde{C}$-bi-Lipschitz equivalent to $\cR := [0,|A-C|] \times [0, |A-D|\sin\beta]$,  then there exists a $\tilde{C}$-bi-Lipschitz finitely piecewise affine homeomorphism $\Phi : \mathbb{P} \to \cR$.   

We can apply Corollary \ref{coroll: non optimal extension} to $\psi:= \tilde{\varphi} \circ \Phi^{-1} : \partial \cR \to \R^2$ to get a finitely piecewise affine homeomorphism $\tilde{\omega} : \cR \to \R^2$ coinciding with $\psi$ on $\partial \cR$ such that 
\begin{equation*}
 \| D\tilde{\omega}  \|_{L^1 (\cR)} \leq   \tilde{C}\mathcal{H}^1(\partial \cR) \mathcal{H}^1(\psi(\partial \cR)).  
\end{equation*}
Then the map $\omega := \tilde{\omega} \circ \Phi : \mathbb{P} \to \R^2$ is a finitely piecewise affine homeomorphism coinciding with $\tilde{\varphi}$ on $\partial \mathbb{P}$  and satisfying 
\[   \| D\omega  \|_{L^1 (\mathbb{P})} \leq \tilde{C} \| D\tilde{\omega}  \|_{L^1 (\cR)} \leq \tilde{C} \mathcal{H}^1(\Phi(\partial \mathbb{P})) \mathcal{H}^1 (\tilde{\varphi} \circ \Phi^{-1} (\partial \cR)) \leq \tilde{C} \mathcal{H}^1(\partial \mathbb{P}) \mathcal{H}^1(\tilde{\varphi} (\partial \mathbb{P})).  \]

Once here we recall \eqref{eq: length of Gamma} and the first of \eqref{eq: estimates from Step I} to get that 
\begin{align*}
 \mathcal{H}^1( \partial \mathbb{P}) &\leq |A-C| + |D-A| + |C-E| + \mathcal{H}^1(\Gamma) + |E-G| \\
 &\leq  4d + 2\eta,
\end{align*}
while from ii),  the fact that $\tilde{\varphi}(\Gamma \cup EG) = \varphi^1(DE) = \tilde{\gamma}_{\varphi(D)\varphi(E)}$ and \eqref{eq: estimate path from varphiD to varphiE}, we deduce 
\begin{align*}
  \mathcal{H}^1 \big(  \tilde{\varphi}(\partial \mathbb{P}) \big) &\leq \mathcal{H}^1(\varphi(AC)) +  \mathcal{H}^1(\varphi(AD)) +  \mathcal{H}^1(\varphi(CE)) +  \mathcal{H}^1(\varphi^1(ED)) \\
  &\leq |\varphi(A) - \varphi(C)| + |\varphi(A) - \varphi(D)| + |\varphi(C) - \varphi(E)| + \mathcal{H}^1(\tilde{\gamma}_{\varphi(D)\varphi(E)}) \\
 &\leq 2\bd + 12\eta.
\end{align*}

Then the last two observations together with the estimate on $D\omega$ imply

\begin{equation}\label{eq: estimate of the extension inside the bad polygon}
 \| D\omega \|_0 (\mathbb{P}) \leq \tilde{C} \| D\omega  \|_{L^1 (\mathbb{P})} \leq \tilde{C}(d+\eta) (\bd + \eta) \leq \tilde{C}d\bd.
\end{equation}

We finally define $v : T \to \R^2$ the finitely piecewise affine map such that 
\[  v _{\rceil \mathbb{P}} = \omega, \quad v _{\rceil \tilde{T}} = w \quad \text{and} \quad   v_{\rceil R_i} = v_i  \,\,\, \text{for every}\,\, i = 0 \ldots M-1 .  \]
We observe that $v$ is continuous and injective, hence a homeomorphism, since $v = \tilde{\varphi}$ on $\Upsilon$.  From the same observation we also deduce that $v = \varphi$ on $\partial T$ because $\tilde{\varphi} = \varphi$ there.   Gathering together \eqref{eq: estimate of the extension inside rectangle}, \eqref{eq: estimate of the extension inside small triangle},  \eqref{eq: estimate of the extension inside the bad polygon} and \eqref{eq: estimate Psi function on rectangles} we find 
\begin{align*}
\| Dv \|_0 (T) &= \| D\omega \|_0 (\mathbb{P}) + \| Dw \|_0 (\tilde{T}) + \sum_{i=0}^{M-1} \| Dv_i \|_0 (R_i) +  \| Dv \|_0 (\hat{T}) \\
&\leq \tilde{C}d\bd + \eta^2 \left( 1 + \frac{1}{2\tan\beta} \right) + \sum_{i=0}^{M-1} \left( \Psi_0(\tilde{\varphi} _{\rceil \partial R_i}) + \frac{\eps}{2^i} \right)  + d\epsilon\\
&\leq  \tilde{C}d\bd + \eta^2 \left( 1 + \frac{1}{2\tan\beta} \right) + \Psi_0(\varphi) +C\bd \mathcal{H}^1(\partial T) +\tilde{C} \eps  \\
&\leq  \Psi_0(\varphi) + \tilde{C}\mathcal{H}^1(\partial T) \mathcal{H}^1(\varphi(AC)) + \tilde{C}\eps,
\end{align*}
where in the last inequality we used that $d = |A-C| \leq \mathcal{H}^1(\partial T)$ and $\bd = |\varphi(A - \varphi(C))|$.
\end{proof}

\section{Proof of Theorem \ref{thm: rotated minimal extension} and Theorem~\ref{Componentwise ExtensionAlt}}
\begin{proof}[Proof of Theorem \ref{thm: rotated minimal extension}]
Let $\eps >0$ be arbitrary fixed and let
\begin{equation}\label{eq:choice of eta}
0 < \eta < \min \left \lbrace  1, \frac{\eps}{\tilde{C} + \tilde{C} \mathcal{H}^1(\Q)} \right \rbrace
\end{equation} 
for some large appropriate but fixed geometric constant $\tilde{C}$. We describe in detail the proof when $\Q$ is of class $iii)$ in the sense of Remark \ref{rmk: class I) and II)}, while the other cases are an obvious modification of the current argument. 

Applying Lemma \ref{lemma: skeleton triangles} to $\Q, \phi,\alpha$ and the parameter $\eta$,  we can partition $\Q$ in two triangles $T_1,  T_2$ and a convex polygon $\Delta$ and find a continuous, piecewise linear, injective map $\bar{\phi} : \partial T_1 \cup  \partial T_2 \cup  \partial\Delta \to \R^2$ with the properties listed in the statement of Lemma \ref{lemma: skeleton triangles}. In particular, from \eqref{GeoTry} it follows that 
\begin{equation}\label{I}
 \Psi_\alpha (\bar{\phi} _{\rceil \partial \Delta}) \leq \Psi_\alpha (\phi) + \eta
\end{equation}
where \eqref{SmallTry} and \eqref{SmallImageTry} ensure that 
\begin{equation}\label{II}
\mathcal{H}^1(\partial T_1) + \mathcal{H}^1(\partial T_2) < \eta, \qquad \mathcal{H}^1(\bar{\phi}(\partial T_1)) +  \mathcal{H}^1(\bar{\phi}(\partial T_2)) < \eta.
\end{equation}

\smallskip
Furthermore, since $\Delta$ is a convex polygon with two parallel sides in direction $\alpha$,  we can apply the $\alpha$-rotated version of Lemma \ref{lemma: skeleton strips} to $\Delta,$ $\bar{\phi}^1$ and the parameter $\eta$ so to find $M$ increasing values $(t_i)_{i=0}^{M-1}$ and $\alpha$-rotated strips $S_i$,  which can be seen as the union of a rectangle $R_i$ and two triangles $T_i^1$ and $T_i^2$,  and a continuous piecewise linear injective map $\hat{\phi}: \bigcup_{i=0}^{M-1} \partial T_i^1 \cup \partial R_i \cup \partial T_i^2 \to \R^2$ coinciding with $\bar{\phi}$ on $\partial \Delta$ with the properties of Lemma  \ref{lemma: skeleton strips}. In particular, thanks to \eqref{eq: estimate psi functions on skeleton} we deduce  
\begin{equation}\label{III}
\sum_{i=0}^{M-1} \Big( \Psi_\alpha (\hat{\phi} _{\rceil \partial T_i^1}) + \Psi_\alpha (\hat{\phi} _{\rceil \partial R_i}) + \Psi_\alpha (\hat{\phi} _{\rceil \partial T_i^2})  \Big) \leq \Psi_\alpha (\bar{\phi}^1 _{\rceil \partial \Delta}) + \eta,
\end{equation}
while \eqref{eq: estimate boundary segments} ensures that 
\begin{equation}\label{IV}
\mathcal{H}^1 (\hat{\phi}(I_i^1)) + \mathcal{H}^1 (\hat{\phi}(I_i^2)) < \eta
\end{equation}
where $I_i^1 = \partial T_i^1 \cap \partial\Delta$ and $I_i^2 = \partial T_i^2 \cap \partial\Delta$. 

We are finally in position to define a function 
\[\tilde{\phi} : \partial T_1 \cup \left(\bigcup_{i=0}^{M-1} \partial T_i^1 \cup \partial R_i \cup \partial T_i^2\right) \cup \partial T_2 \to \R^2  \]
that is continuous, injective, finitely piecewise linear and such that $\tilde{\phi} = \bar{\phi}^1$ on $\partial T_1 \cup \partial T_2$ and $\tilde{\phi} = \hat{\phi}$ on $\bigcup_{i=0}^{M-1} \partial T_i^1 \cup \partial R_i  \cup \partial T_i^2$. 

Once here we will perform the homeomorphic piecewise affine extension on $T_1, T_2, T_i^1, T_i^2$ and $R_i$ independently for every $i = 0 \dots M-1$.
We first focus on the extension inside the strips $S_i$.  Let $i \in \{ 0, \ldots M-1 \}$ be fixed, we then apply the $\alpha$-rotated version of Proposition \ref{prop: straight minimal extension} to $R_i, \tilde{\phi} _{\rceil \partial R_i}$ and parameter $\eta(t_{i+1} - t_i)$ to find finitely piecewise affine homeomorphisms $v_i : R_i \to \R^2$ coinciding with $\tilde{\phi}$ on $\partial R_i$ such that 
\begin{equation}\label{V}
\|  Dv_i \|_\alpha (R_i) \leq \Psi_\alpha (\tilde{\phi} _{\rceil \partial R_i}) + \eta(t_{i+1} - t_{i}). 
\end{equation}
By construction, we have that $T_i^{1}, T_i^{2}$ are right-angle triangles whose hypotenuse is contained in $\partial \Delta \cap \partial \Q$,  we can apply the $\alpha$-rotated version of Lemma~\ref{lemma: extension left & right triangles} to $T_i^{1,2}, \tilde{\phi} _{\rceil \partial T_i^{1,2}}$ and parameter $\eta(t_{i+1} - t_i)$ to find finitely piecewise affine homeomorphisms $w_i^{1,2} : T_i^{1,2} \to \R^2$ coinciding with $\tilde{\phi}$ on $\partial T_i^{1,2}$ such that 
\begin{equation*}
\|  Dw_i^{1,2} \|_\alpha (T_i^{1,2}) \leq \Psi_\alpha(\tilde{\phi} _{\rceil \partial T_i^{1,2}}) + \tilde{C}\mathcal{H}^1(\partial T_i^{1,2}) \mathcal{H}^1(\tilde{\phi}(I_i^{1,2})) + \tilde{C}\eta(t_{i+1} - t_i),
\end{equation*}
which thanks to \eqref{IV} implies
\begin{align}\label{VI}
\notag
\|  Dw_i^{1} \|_\alpha (T_i^{1}) + \|  Dw_i^{2} \|_\alpha (T_i^{2})  \leq &\, \Psi_\alpha(\tilde{\phi} _{\rceil \partial T_i^{1}}) + \Psi_\alpha(\tilde{\phi} _{\rceil \partial T_i^{2}}) \\
&+ \big[\tilde{C} ( \mathcal{H}^1(\partial T_i^{1}) + \mathcal{H}^1(\partial T_i^{2})) + \tilde{C}(t_{i+1} - t_i)\big] \eta. 
\end{align} 

Let us also notice that, for future need, since the triangles $T_i^{1,2}$ have one angle equal to $\pi/2$ and, by construction, their hypotenuse is $I_i^{1,2}$, then one has $\mathcal{H}^1(\partial T_i^{1,2}) \leq 3 \mathcal{H}^1(I_i^{1,2})$ for every $i$.  Moreover, being $I_i^{1,2} \subset (\partial \Delta \cap \partial \Q)$ and the triangles pairwise essentially disjoint, we deduce 
\begin{equation}\label{VI bis}
\sum_{i=0}^{M-1} \big(\mathcal{H}^1(\partial T_i^{1}) + \mathcal{H}^1(\partial T_i^{2}))\big) \leq 3\sum_{i=0}^{M-1}  \big( \mathcal{H}^1(I_i^{1}) + \mathcal{H}^1(I_i^{2})\big)   \leq 3 \mathcal{H}^1(\partial \Q).
\end{equation} 

Let us now consider the extension inside the triangles $T_1, T_2$. In this case, by construction, we are in position to apply the $\alpha$-rotated version of Lemma \ref{lemma: extension up & bottom triangles} to $T_{1,2}$ and $\tilde{\phi}_{\rceil \partial T_{1,2}}$ to find bi-affine homeomorphisms $w_{1,2} : T_{1,2} \to \R^2$ such that 
\begin{equation*}
\|  Dw_{1,2} \|_\alpha (T_{1,2}) \leq \mathcal{H}^1 \big(\tilde{\phi}(\partial T_{1,2})\big)  \mathcal{H}^1 \big(\partial T_{1,2}\big) 
\end{equation*}
which, thanks to \eqref{II},  gives
\begin{equation}\label{VII}
\|  Dw_1 \|_\alpha (T_1)  + \|  Dw_2 \|_\alpha (T_2)  <2\eta^2 . 
\end{equation}

We can finally define $v: \Q \to \R^2$ to be the piecewise affine function such that 
\[  v = w_{1,2} \text{ on } T_{1,2}, \quad v = w_i^{1,2} \text{ on } T_i^{1,2} \quad \text{ and } \quad v = v_i \text{ on } R_i  \quad \text{ for every } i = 0 \ldots M-1 . \]
By construction, $v$ is continuous because $v = \tilde{\phi}$ on the one-dimensional skeleton $\partial T_1 \cup \left(\bigcup_{i=0}^{M-1} \partial T_i^1 \cup \partial R_i \cup \partial T_i^2\right) \cup \partial T_2$ and, moreover, $v$ coincides with $\tilde{\phi} = \phi$ on $\partial \Q$. To conclude, it is only left to verify the validity of \eqref{eq: rotated minimal extension}, but this is now a straightforward consequence of \eqref{V}, \eqref{VI}, \eqref{VII},\eqref{III},  \eqref{VI bis} and \eqref{I}.  Indeed,  we get 
$$
	\begin{aligned}
		\| Dv \|_\alpha(\Q) &=   \| Dw_1 \|_\alpha(T_1) + \| Dw_2 \|_\alpha(T_2) + \sum_{i=0}^{M-1} \big(  \| Dw_i^1 \|_\alpha(T_i^1)\\
		&\quad  + \| Dv_i \|_\alpha(R_i) + \| Dw_i^2 \|_\alpha(T_i^2) \big) \\
		& \leq \sum_{i=0}^{M-1} \Big( \Psi_\alpha (\hat{\phi} _{\rceil \partial T_i^1}) + \Psi_\alpha (\hat{\phi} _{\rceil \partial R_i}) + \Psi_\alpha (\hat{\phi} _{\rceil \partial T_i^2})  \Big) \\
		&\,\,+ 4\eta^2 +  \tilde{C}\eta \sum_{i=0}^{M-1} \big[ ( \mathcal{H}^1(\partial T_i^{1}) + \mathcal{H}^1(\partial T_i^{2})) + (t_{i+1} - t_i)\big]  \\
		& \leq \Psi_\alpha (\bar{\phi}^1 _{\rceil \partial \Delta}) + \eta + 4\eta^2 +  \tilde{C}\eta \big( \mathcal{H}^1(\partial \Q) +  \diam \Q \big) \\
		& \leq  \Psi_\alpha (\bar{\phi} _{\rceil \partial \Delta})  + 6\eta  +  \tilde{C}\eta \big( \mathcal{H}^1(\partial \Q)  \big),\\
		&\leq \Psi_\alpha (\phi) +  \tilde{C}\eta  \big( 1  +  \mathcal{H}^1(\partial \Q)  \big),
	\end{aligned}
$$
and then estimate \eqref{eq: rotated minimal extension} follows since $\eta$ has been chosen as in \eqref{eq:choice of eta}.

\end{proof}

\subsection{Proof of Theorem~\ref{Componentwise ExtensionAlt}}$\empty$

To prove the claim it suffices to repeat the proof of the above Lemmas as before but using the estimates from Theorem~\ref{thm:What} instead of from Proposition~\ref{prop: straight minimal extension}. In all of our calculations we estimate $\int \rho_{\P_{\Omega}} (\varphi^1(H_t^2), \varphi^1(H_t^3)) dt$ and $\int_{D_1}^{B_1}  \rho_{\P_{\Omega}} (\varphi^1(V_s^1), \varphi^1(V_s^2)) ds$ separately. Now it suffices to keep them separate instead of summing them.

The key estimates in Lemma~\ref{lemma: skeleton triangles} are \eqref{eq: stima Psi 2/3}, in Lemma~\ref{lemma: skeleton strips} they are \eqref{eq: s belle interne}, \eqref{eq: s brutte interne} and the calculation following; in Lemma~\ref{lemma: extension left & right triangles} are \eqref{Johnny}, \eqref{Jimmy}, \eqref{Jack}.

We can then repeat the proof of Theorem~\ref{thm: rotated minimal extension} with the difference that in \eqref{III}, \eqref{V} and so on we use the separate estimates, rather than the summed estimates expressed using $\Psi_{\alpha}$.
\qed


\begin{thebibliography}{00}

\bibitem{AFP}
\by{\name{Ambrosio}{L.}, \name{Fusco}{N.} and \name{Pallara}{D.}}
\book{Functions of bounded variation and free discontinuity problems}
\publ{Oxford Mathematical Monographs. The Clarendon Press, Oxford University Press, New York, 2000}
\endbook

\bibitem{C}
\by{\name{Campbell}{D.}}
\paper{Diffeomorphic approximation of Planar Sobolev Homeomorphisms in Orlicz-Sobolev spaces}
\jour{J. Funct. Anal.}
\vol{273}
\pages{125--205}
\yr{2017}
\endpaper


\bibitem{CHKR}
\by{\name{Campbell}{D.}, \name{Hencl}{S.} \name{Kauranen}{A.} and \name{Radici}{E.}}
\paper{Strict limits of planar $BV$ homeomorphisms}
\jour{Nonlinear Analysis}
\vol{177} 
\pages{209--237} 
\yr{2018}
\endpaper



\bibitem{CKR}
\by{\name{Campbell}{D.}, \name{Kauranen}{A.} and \name{Radici}{E.}}
\paper{Classification of strict limits of planar $BV$ homeomorphisms}
\jour{arXiv:2101.09013, 2021}
\endprep


\bibitem{CKR2}
\by{\name{Campbell}{D.}, \name{Kauranen}{A.} and \name{Radici}{E.}}
\paper{Classification of area-strict limits of planar $BV$ homeomorphisms}
\jour{arXiv:..., 2022}
\endprep


\bibitem{DP2} 
\by{\name{Daneri}{S.} and \name{Pratelli}{A.}} 
\paper{A planar bi-Lipschitz extension theorem}
\jour{Advances in Calculus of Variations}
\vol{8 \rm no. 3}
\pages{221--266}
\yr{2015}
\endpaper

\bibitem{DP} 
\by{\name{Daneri}{S.} and \name{Pratelli}{A.}} 
\paper{Smooth approximation of bi-Lipschitz orientation-preserving homeomorphisms}
\jour{Ann. Inst. H. Poincar\'e Anal. Non Lin\'eaire}
\vol{31 \rm no. 3}
\pages{567--589}
\yr{2014}
\endpaper

	\bibitem{PP}
	\by{\name{De Philippis}{G.} and \name{Pratelli}{A.}}
	\paper{The closure of planar diffeomorphisms in Sobolev spaces}
	\jour{Ann. Inst. H. Poincar\'e Anal. Non Lin\'eaire}
	\vol{37, \rm no. 1}
	\pages{181-224}
	\yr{2020}
	\endpaper


\bibitem{HP} 
\by{\name{Hencl}{S.} and \name{Pratelli}{A.}}
\paper{Diffeomorphic Approximation of $W^{1,1}$ Planar Sobolev Homeomorphisms}
\jour{J. Eur. Math. Soc}
\vol{20, \rm no. 3}
\pages{597--656}
\yr{2018}
\endpaper


\bibitem{IKO1}
\by{\name{Iwaniec}{T.}, \name{Kovalev}{L.} and \name{Onnien}{J.}}
\paper{Diffeomorphic approximation of Sobolev homeomoprhisms}
\jour{Arch. Rational Mech. Anal}
\vol{201, \rm no. 3} 
\pages{1047--1067} 
\yr{2011}
\endpaper


\bibitem{IKO2}
\by{\name{Iwaniec}{T.}, \name{Kovalev}{L.} and \name{Onnien}{J.}}
\paper{Hopf differentials and smoothing Sobolev homeomorphisms}
\jour{International Mathematics Research Notices}
\vol{14} 
\pages{3256--3277} 
\yr{2012}
\endpaper


\bibitem{IO}
\by{\name{Iwaniec}{T.} and \name{Onninen}{J.}}
\paper{Limits of Sobolev homeomorphisms}
\jour{J. Eur. Math. Soc}
\vol{19, \rm no. 2} 
\pages{473--505} 
\yr{2017}
\endpaper



\bibitem{BiP}
\by {\name{Pratelli}{A.}}
\paper{On the bi-Sobolev planar homeomorphisms and their approximation}
\jour{Nonlinear Analysis: Theory, Methods and Applications}
\vol{154} 
\pages{258--268} 
\yr{2017}
\endpaper



\bibitem{PR2}
    \by{\name{Pratelli}{A.} and \name{Radici}{E.}}
    \paper{On the planar minimal BV extension problem}
    \jour{Rendiconti Lincei: Matematica e Applicazioni}
    \vol{29, \rm no. 3}
	\pages{511--555}
	\yr{2018}
    \endpaper


\bibitem{PR3}
    \by{\name{Pratelli}{A.} and \name{Radici}{E.}}
    \paper{Approximation of planar BV homeomorphisms by diffeomorphisms}
    \jour{J. Func. Anal}
    \vol{276}	
	\pages {659--686}
    \yr{2019}
    \endpaper
    
    \bibitem{R}
    \by{\name{Radici}{E.}}
    \paper{A planar Sobolev extension theorem for piecewise linear homeomorphisms}
    \jour{Pacific Journal of Mathematics}
    \vol{238 \rm no. 2}	
    \pages {405--418}
    \yr{2016}
    \endpaper


\end{thebibliography}
\end{document}